\documentclass{amsart}

\usepackage{fullpage}
\usepackage{amssymb}
\usepackage{amsfonts}
\usepackage{amsthm}
\usepackage{amsmath}
\usepackage{mathtools}
\usepackage{bbm}
\usepackage{bussproofs}
\usepackage{mathrsfs}
\usepackage{tikz, tikz-cd}
\usepackage[all,2cell]{xy}
\usepackage{epigraph}
\usepackage{ stmaryrd } %For power series ring brackets
\usepackage{todonotes}

\tikzset{%
	symbol/.style={%
		draw=none,
		every to/.append style={%
			edge node={node [sloped, allow upside down, auto=false]{$#1$}}}
	}
}

\usetikzlibrary{matrix,arrows}

\newtheorem{Theorem}{Theorem}

\newtheorem{conjecture}[Theorem]{Conjecture}

\newtheorem{proposition}[Theorem]{Proposition}
\newtheorem{lemma}[Theorem]{Lemma}
\newtheorem{corollary}[Theorem]{Corollary}

\theoremstyle{definition}
\newtheorem{example}[Theorem]{Example}
\newtheorem{remark}[Theorem]{Remark}
\newtheorem{definition}[Theorem]{Definition}

\DeclareMathOperator{\Cscr}{\mathscr{C}}
\DeclareMathOperator{\Dscr}{\mathscr{D}}
\DeclareMathOperator{\Escr}{\mathscr{E}}

\DeclareMathOperator{\Uscr}{\mathscr{U}}

\DeclareMathOperator{\Abb}{\mathbb{A}}

\DeclareMathOperator{\Rbb}{\mathbb{R}}
\DeclareMathOperator{\Sbb}{\mathbb{S}}
\DeclareMathOperator{\Tbb}{\mathbb{T}}

\DeclareMathOperator{\mfrak}{\mathfrak{m}}

\DeclareMathOperator{\Xfrak}{\mathfrak{X}}

\DeclareMathOperator{\Yfrak}{\mathfrak{Y}}

\DeclareMathOperator{\yon}{\mathbf{y}}

\DeclareMathOperator{\Tan}{\mathbf{Tan}}
\DeclareMathOperator{\fTan}{\mathfrak{Tan}}
\DeclareMathOperator{\QCoh}{\mathbf{QCoh}}

\DeclareMathOperator{\id}{id}

\DeclareMathOperator{\Set}{\mathbf{Set}}

\DeclareMathOperator{\R}{\mathbb{R}}

\DeclareMathOperator{\N}{\mathbb{N}}

\DeclareMathOperator{\KAlg}{\mathnormal{K}-\mathbf{Alg}}

\DeclareMathOperator{\FSch}{\mathbf{FSch}}

\DeclareMathOperator{\Q}{\mathbb{Q}}

\DeclareMathOperator*{\colim}{colim}
\DeclareMathOperator{\Ind}{Ind}

\DeclareMathOperator{\Spec}{Spec}
\DeclareMathOperator{\Sch}{\mathbf{Sch}}

\DeclareMathOperator{\Free}{Free}
\DeclareMathOperator{\Forget}{Forget}
\DeclareMathOperator{\Cat}{\mathbf{Cat}}

\DeclareMathOperator{\Sym}{Sym}

\DeclareMathOperator{\Spf}{Spf}

\DeclareMathOperator{\op}{op}

\DeclareMathOperator{\fCat}{\mathfrak{Cat}}

\let\emptyset\varnothing
\let\epsilon\varepsilon

\newcommand{\Dbeqstack}[1]{\mathnormal{D}^{\mathnormal{b}}_{\text{eq}}(\underline{[\mathnormal{G} \backslash \mathnormal{X}]}_{\bullet};\overline{\Q}_{\ell})}

\DeclareMathOperator{\lin}{\operatorname{lin}}
\DeclareMathOperator{\Dlin}{\mathnormal{D}\operatorname{lin}}
\DeclareMathOperator{\Weil}{\mathbf{Weil}}
\DeclareMathOperator{\APoly}{\mathnormal{A}-\mathbf{Poly}}
\DeclareMathOperator{\AMod}{\mathnormal{A}-\mathbf{Mod}}
\DeclareMathOperator{\AModfd}{\AMod_{f.d.}}
\DeclareMathOperator{\Smooth}{\R-\mathbf{Smooth}}
\DeclareMathOperator{\RVectfd}{\R-\mathbf{Vec}_{f.d.}}
\DeclareMathOperator{\RVect}{\R-\mathbf{Vec}}

\let\emptyset\varnothing
\let\epsilon\varepsilon

\newcommand{\pullbackcorner}[1][dr]{\save*!/#1-1.2pc/#1:(-1,1)@^{|-}\restore}

\newcommand{\ul}[1]{\underline{{#1}}}
\newcommand{\Zar}[1]{\mathbf{Zar}({#1})}

\UseTwocells

\numberwithin{Theorem}{section}
\usepackage{hyperref} 

\title{Tangent Ind-Categories}
\author{Geoff Vooys}
\date{\today}

\begin{document}

\begin{abstract}
In this paper we show that if $\Cscr$ is a tangent category then the Ind-category $\Ind(\Cscr)$ is a tangent category as well with a tangent structure which locally looks like the tangent structure on $\Cscr$. Afterwards we give a pseudolimit description of $\Ind(\Cscr)_{/X}$ when $\Cscr$ admits finite products, show that the $\Ind$-tangent category of a representable tangent category remains representable (in the sense that it has a microlinear object), and we characterize the differential bundles in $\Ind(\Cscr)$ when $\Cscr$ is a Cartesian differential category. Finally we compute the $\Ind$-tangent category for the categories $\mathbf{CAlg}_{A}$ of commutative $A$-algebras, $\Sch_{/S}$ of schemes over a base scheme $S$, $\APoly$ (the Cartesian differential category of $A$-valued polynomials), and $\Smooth$ (the Cartesian differential category of Euclidean spaces). In particular, during the computation of $\Ind(\Sch_{/S})$ we give a definition of what it means to have a formal tangent scheme over a base scheme $S$.
\end{abstract}


\maketitle
\tableofcontents

\section{Introduction}

This paper is a study of two different categorical geometric concepts and the ways in which they interact with each other. On one hand, we have the categorification of tangent bundles and differential geometry (as discovered by Rosick{\'y} in \cite{Rosicky} and later rediscovered by Cockett and Crutwell in \cite{GeoffRobinDiffStruct}) and the categorical approach to differential reasoning by the semantics of the tangent bundle functor itself; on the other hand we have $\Ind$-categories (as discovered by Grothendieck in \cite{SGA4} in its relation to presheaf toposes the mechanics involved in taking filtered colimits) which give an inductive cocompletion of a category and allow for well-behaved ``at the limit'' geoemtric arguments. In particular, $\Ind$-completions and $\Ind$-arguemnts often carry local, nearly infinitesimal, information and behave as though they are completions of spaces in certain adic topologies. Because of the geometric similarities between studying infinitesimal neighborhoods of successive tangent vectors and their geometry together with the completions of a space in an adic topology involving tangent vectors, in this paper we examine how these categorical structures interact. Let us first recall in more detail how tangent categories are incarnated.

On one hand, tangent categories are a categorical generalization of the tangent bundle of a (smooth) manifold and provide a category-theoretic setting in which to do differential geometry. Originally discovered by Rosick{\'y} in \cite{Rosicky} and later rediscovered and refined by Cockett and Crutwell in \cite{GeoffRobinDiffStruct}, tangent categories have given connections between differential geometry, higher category theory, (de Rham) cohomology, linear logic, programming, machine learning, and have become ubiquitous in both theoretical computer science and pure mathematics. These have been studied in detail in many of these different contexts, and in \cite{GarnerEmbedding} we even see some interactions of enriched tangent structures. Moreover, work of Crutwell and Lemay (cf.\@ \cite{GeoffJS}) has recently shown that the category $\Sch_{/S}$ of schemes over a base scheme $S$ admits a tangent structure as well. 

On the other hand, ind-categories arise as the free filtered cocompletion of a category $\Cscr$ and were first studied in detail by Grothendieck in \cite{SGA4}. These categories are important in algebraic geoemtry, as by \cite{Strickland} there is a close relationship with the categories $\FSch_{/S}$ and $\Ind(\Sch_{/S})$ where $\FSch_{/S}$ is the category of formal schemes over $S$. This means in particular that ind-categories can be thought of as formal completions of objects along some closed subobject. We then can think of each object in $\Ind(\Cscr)$, by way of formal analogy, as an object which is (adically) completed along some closed subobject and hence admits all formal infinitesimal neighborhoods along that subobject. Our intuition as we proceed is to study these formal infinitesimal neighborhoods in these objects when they \emph{also} have a notion of tangent in the following sense: if we have infinitesimal neighborhoods and we have tangent vectors, then we should expect to be able to complete the tangent space to allow infinitesimal neighborhoods of the tangent vectors themselves. Note that this sits in contrast to the fact that if we take the free cocompletion of the category $\Cscr$ (as modeled, for instance, by the presheaf topos $[\Cscr^{\op},\Set]$), we cannot expect $[\Cscr^{\op},\Set]$ to be a tangent category even when $\Cscr$ is --- cf.\@ \cite{GeoffRobinDiffStruct} for details. The difference here is that our infinitesimal neighborhoods are sufficiently well behaved (via their filtrations) whereas the free cocompletion has colimits which are too wild to witness  and interact well with the categorical limits that describe tangent-theoretic information.

\subsection{Results and Structure of the Paper}
We begin this paper with a quick review of tangent categories, their morphisms, and the  category (and $2$-category) of tangent categories. Afterwards we familiarize ourselves with the ind-construction and get to know the $\Ind$ pseudofunctor $\Ind:\fCat \to \fCat$. We then study the interaction of the tangent structures as they pass through the $\Ind$ pseudofunctor and lift to the ind-case. This culminates in our first big result, which is Theorem \ref{Thm: Ind Tangent Category}.
\begin{Theorem}[{cf.\@ Theorem \ref{Thm: Ind Tangent Category}}]\label{Thm: Ind tangent cat but intro version}
	Let $(\Cscr,\Tbb)$ be a tangent category. Then the category $(\Ind(\Cscr),\Ind(\Tbb))$ is a tangent category where $\Ind(\Tbb)$ is the tangent structure
\[
\Ind(\Tbb) := (\Ind(T), \hat{p}, \Ind(0), \hat{+}, \hat{\ell}, \hat{c})
\]
where $\Ind(T)$ is the indicization of $T$ induced by Proposition \ref{Prop: Ind functors} and $\hat{p}, \Ind(0), \hat{+}, \hat{\ell},$ and $\hat{c}$ are the natural transformations constructed in Lemmas \ref{Lemma: Existence of Indp and Ind0}, \ref{Lemma: Existence of ind-bundle addition}, \ref{Lemma: Existence of Ind-vertical lift}, and Lemma \ref{Lemma: Existence of ind-canonical flip}, respectively.
\end{Theorem}
After establishing this we show a functoriality result in Theorem \ref{Thm: Functoriality of ind-category} (which, as a corollary, allows us to deduce that $\Ind:\fCat \to \fCat$ restricts to a pseudofunctor $\Ind:\fTan \to \fTan$). 
\begin{Theorem}[{cf.\@ Theorem \ref{Thm: Functoriality of ind-category}}]\label{Thm: Ind tangent functoriality but intro version}
Let $(F,\alpha):(\Cscr, \Tbb) \to (\Dscr, \Sbb)$ be a morphism of tangent categories. Then the induced map $(\Ind(F), \hat{\alpha}):(\Ind(\Cscr), \Ind(\Tbb)) \to (\Ind(\Dscr), \Ind(\Sbb))$ is a morphism of tangent categories where $\hat{\alpha}$ is the natural transformation:
\[
\begin{tikzcd}
\Ind(\Cscr) \ar[rr, bend left = 105, ""{name = UU}]{}{\Ind(F) \circ \Ind(T)} \ar[rr, bend left = 25, ""{name = U}]{}[description]{\Ind(F \circ T)} \ar[rr, bend right = 15, ""{name = L}]{}[description]{\Ind(S \circ F)} \ar[rr, bend right = 70, swap, ""{name = LL}]{}{\Ind(S) \circ \Ind(F)} & & \Ind(\Dscr) \ar[from = UU, to = U, Rightarrow, shorten >= 4pt, shorten <= 4pt]{}{\phi_{T, F}} \ar[from = U, to = L, Rightarrow, shorten >= 4pt, shorten <= 4pt]{}{\Ind(\alpha)} \ar[from = L, to = LL, Rightarrow, shorten >= 4pt, shorten <= 4pt]{}{\phi_{F, S}^{-1}}
\end{tikzcd}
\]
Furthermore, $(\Ind(F), \hat{\alpha})$ is a strong tangent morphism if and only if $(F, \alpha)$ is a strong tangent morphism.
\end{Theorem}

We also show how the $\Ind$ construction interacts with the forgetful functor $\Forget:\fTan \to \fCat$ and the free tangent functor $\Free:\fCat \to \fTan$. More explicitly we show that the diagram
\[
\begin{tikzcd}
\fTan \ar[r]{}{\Ind} \ar[d, swap]{}{\Forget} & \fTan \ar[d]{}{\Forget} \\
\fCat \ar[r, swap]{}{\Ind} & \fCat
\end{tikzcd}
\]
commutes strictly but in the other case that there is a pseudonatural transformation
\[
\begin{tikzcd}
\fCat \ar[d, swap]{}{\Free} \ar[r, ""{name = U}]{}{\Ind} & \fCat \ar[d]{}{\Free} \\
\fTan \ar[r, swap, ""{name = D}]{}{\Ind} & \fTan \ar[from = U, to = D, Rightarrow, shorten <= 4pt, shorten >= 4pt]{}{\alpha}
\end{tikzcd}
\]
which cannot be a pseudonatural equivalence. In particular, we show that $\Ind$ does not commute with $\Free$, even up to pseudonatural equivalence.

After presenting these theorems we show how the $\Ind$-tangent structure interacts with various constructions which are important in the tangent category literature. In particular, we show that the $\Ind$-tangent category of a representable tangent category remains representable and then determine when, for a Cartesian differential category $\Cscr$, an object in $\Ind(\Cscr)$ is a differential bundle over the terminal object $\top_{\Ind(\Cscr)}$. As a corollary we give a necessary and sufficient condition for recognizing when, for a Cartesian differential category $\Cscr$, $\Ind(\Cscr)$ is not a Cartesian differential category.

\begin{proposition}[{cf.\@ Proposition \ref{Prop: Reple internal hom of inding things}}]
Let $\Cscr$ be a Cartesian closed tangent category. If $\Cscr$ is representable then there is an object $\ul{D}$ in $\Ind(\Cscr)$ for which $\Ind(T) \cong [\ul{D},-]$.
\end{proposition}
\begin{proposition}[{cf.\@ Proposition \ref{Prop: Ind of CDC has all diff objs means C is Dlinear}}]
Let $\Cscr$ be an $A$-linear Cartesian differential category for a commutative rig $A$. Then $\Cscr = \Cscr_{\Dlin}$ if and only if every object in $\Ind(\Cscr)$ is a differential bundle over the terminal object $\top_{\Ind(\Cscr)}$.
\end{proposition}

We close this paper with a review of the tangent structure on the categories $\mathbf{CAlg}_{A}$ of commutative algebras over a commutative rig $A$; $\Sch_{/S}$ over a base scheme $S$; the $A$-linear Cartesian differential category $\APoly$ of polynomials over, again, a commutative rig $A$; and then the Cartesian differential category $\Smooth$ of Euclidean spaces and smooth maps between them. Of particular interest are a definition of a formal tangent scheme (together with an explicit calculation) together with the following characterizations of the differential objects in $\Ind(\APoly)$ and $\Ind(\Smooth)$.
\begin{Theorem}[{cf.\@ Theorem \ref{Thm: Diff Obs in APoly are Modules}}]
	The category of differential objects in $\Ind(\APoly)$ together with linear bundle maps between them is equivalent to $\Ind(\AModfd)$. In particular, there is an equivalence of categories of $\mathbf{Diff}(\Ind(\APoly))$ and the category $\AMod$ of $A$-modules.
\end{Theorem}
\begin{Theorem}[{cf.\@ Theorem \ref{Thm: Diff objects in IndSmooth are Real Vecs}}]
The category of differential objects in $\Ind(\Smooth)$ together with linear bundle maps between them is equivalent to $\Ind(\RVectfd)$. In particular, there is an equivalence of categories between $\mathbf{Diff}(\Ind(\Smooth))$ and the category $\RVect$ of $\R$ vector spaces.
\end{Theorem}

%While tangent categories have not yet become prominent in algebro-geometric settings, in this paper we will show that an important construction in algebro-geometric and topological algebraic applications, the $\Ind$-construction, behaves well with the differential geometry induced by a tangent category by showing that the ind-category of a tangent category is canonically a tangent category $(\Cscr, \Tbb)$ itself carrying ind-tangents. In light of the equivalence of categories $\Ind(\Sch_{/S}) \simeq \FSch_{/S}$ this allows us to think of 

\subsection*{Acknowledgments}
I would like to thank JS Lemay for many interesting conversations regarding this material and suggesting many of the directions in which to examine the $\Ind$-construction and the $\Ind$-tangent structure\footnote{Much of the material in Sections 4.2, 4.3, 5.1, 5.3, and 5.4 was examined and developed following our conversations.}. I'd also like to thank JS, Rick Blute, and Dorette Pronk for reading early drafts of this paper and providing helpful suggestions. Additionally I'd like to thank Rory Lucyshyn-Wright for interesting conversations regarding the $\Ind$-pseudofunctor. Some of this work was presented at the 2023 Foundational Methods in Computer Science conference, and I'm grateful to the organizers for the hospitality and intellectually stimulating environment which allowed me to produce Sections 4 -- 5 of this paper.

\section{A Review of Tangent Categories}
Let us recall what it means for a category to be a tangent category following R.\@ Cockett and G.\@ Crutwell in \cite{GeoffRobinDiffStruct}. These categories are an abstraction of what it means to have a category and tangent bundle functor, together with all the relations and structure we require/expect of such a bundle object. In what follows below, if $\alpha:FX \to X$ is a natural transformation of an endofunctor $F:\Cscr \to \Cscr$, we write $F_2X$ for the pullback
\[
\xymatrix{
F_2X \ar[r]^-{\pi_1} \ar[d]_{\pi_2} \pullbackcorner & FX \ar[d]^{\alpha_X} \\
FX \ar[r]_-{\alpha_X} & X
}
\]
if it exists in $\Cscr$.

\begin{definition}[{\cite[Definition 2.3]{GeoffRobinDiffStruct}}]\label{Defn: Tangent Categpry}
A category $\Cscr$ has a tangent structure $\Tbb = (T,p,0,+,\ell,c)$ when the following six axioms hold:
\begin{enumerate}
	\item $T:\Cscr \to \Cscr$ is a functor equipped with a natural transformation $p:T \Rightarrow \id_{\Cscr}$ such that for every object $X$ in $\Cscr$ all pullback powers of $p_X:TX \to X$ exist and for all $n \in \N$ the functors $T^n$ preserve these pullback powers.
	\item There are natural transformations $+:T_2 \Rightarrow T$ and $0:\id_{\Cscr} \to T$ for which each map $p_X:TX \to X$ describes an additive bundle in $\Cscr$, i.e., $p:TX \to X$ is an internal commutative monoid in $\Cscr_{/X}$ with addition and unit given by $+$ and $0$, respectively.
	\item There is a natural transformation $\ell:T \Rightarrow T^2$ such that for any $X \in \Cscr_0$, the squares
%	\[
%	\xymatrix{
%	TX \ar[r]^-{\ell_X} \ar[d]_{p_X} & T^2X \ar[d]^{Tp_{X}} \\
%	X \ar[r]_-{0_X} & TX
%	}
%	\]
	\[
	\begin{tikzcd}
	TX \ar[r]{}{\ell_X} \ar[d, swap]{}{p_X} & T^2X \ar[d]{}{Tp_X} & TX \times_{X} TX \ar[d, swap]{}{+_X} \ar[rr]{}{\langle\ell_X\circ \pi_1, \ell_X\circ \pi_2\rangle} & & T^2X \times_{TX} T^2X \ar[d]{}{+_{TX}} & X \ar[d, swap]{}{0_X} \ar[r]{}{0_X} & TX \ar[d]{}{0_{TX}} \\
	X \ar[r, swap]{}{0_X} & TX & TX  \ar[rr, swap]{}{\ell_X} & & T^2X & TX \ar[r, swap]{}{\ell_X} & T^2X
	\end{tikzcd}
	\]
	all commute, i.e., $(\ell_X, 0_X)$ is a morphism of bundles in $\Cscr$ (cf.\@ \cite[Definition 2.2]{GeoffRobinDiffStruct}). %\todo{Fix the top horizontal map in the middlle diagram} %	\footnote{Because $T$ preserves all pullback powers of the tangent functor, $Tp:T^2X \to TX$ is a commutative monoid in $\Cscr_{/X}$ with addition and unit induced by $T+:T(T_2X) \to TX$ and $T0_X:TX \to T^2X$ after doing some pre-and-post-composition with the isomorphism $T(T_2X) \cong T_2(TX)$ and/or $X \xrightarrow{0_X} TX \xrightarrow{T0_X} T^2X$.}
	\item There is a natural transformation $c:T^2 \Rightarrow T^2$ such that for all $X \in \Cscr_0$ the squares
%	\[
%	\xymatrix{
%	T^2X \ar[r]^-{c_X} \ar[d]_{Tp_X} & T^2X \ar[d]^{p_{TX}} \\
%	TX \ar[r]_-{\id_{TX}} & TX
%	}
%	\]
	\[
	\begin{tikzcd}
	T^2X \ar[r]{}{c_X} \ar[d, swap]{}{Tp_X} & T^2X \ar[d]{}{p_{TX}} & T^2X \times_{TX} T^2X \ar[d, swap]{}{(T \ast +)_{X}} \ar[rr]{}{\langle c_X \circ \pi_1, c_X\circ \pi_2\rangle} & & T^2X \times_{TX}T^2X \ar[d]{}{(+ \ast T)_X} & TX \ar[d, swap]{}{(T \ast 0)_X} \ar[r]{}{\id_{TX}} & TX \ar[d]{}{(0 \ast T)_X} \\
	TX \ar[r, swap]{}{\id_{TX}} & TX & T^2X \ar[rr, swap]{}{c_X} & & T^2X & T^2X \ar[r, swap]{}{c_X} & T^2X
	\end{tikzcd}
	\]
	commute, i.e., $(c_X,\id_{TX})$ describes a bundle morphism in $\Cscr$%\todo{Fix the top horizontal arrow in the middle diagram}.
	\item We have the equations $c^2 = \id_{\Cscr}, c \circ \ell = \ell$, and the diagrams
	\[
	\begin{tikzcd}
	T \ar[r]{}{\ell} \ar[d, swap]{}{\ell} & T^2 \ar[d]{}{T \ast \ell} & T^3 \ar[d, swap]{}{c \ast T} \ar[r]{}{T \ast c} & T^3 \ar[r]{}{c \ast T} & T^3\ar[d]{}{c \ast T} \\
	T^2 \ar[r, swap]{}{\ell \ast T} & T^3 & T^3 \ar[r, swap]{}{T \ast c} & T^3 \ar[r, swap]{}{c \ast T} & T^3
	\end{tikzcd}
	\begin{tikzcd}
	T^2 \ar[d, swap]{}{c} \ar[r]{}{\ell \ast T} & T^3 \ar[r]{}{T \ast c} & T^3 \ar[d]{}{c \ast T} \\
	T^2 \ar[rr, swap]{}{T \ast \ell} & & T^3
	\end{tikzcd}
	\]
	commute in the functor category $[\Cscr,\Cscr]$ where $F \ast \alpha$ and $\alpha \ast F$ denote the left and right whiskering of a natural transformation $\alpha$ by a functor $F$, respectively.
	\item The diagram
	\[
	\begin{tikzcd}
	T_2X \ar[rrrr]{}{(T \ast +)_X \circ \langle \ell \circ \pi_1, 0_{TX} \circ\pi_2 \rangle} & &	& & T^2X \ar[rrr, shift left = 0.5ex]{}{T(p_X)} \ar[rrr, swap, shift left = -0.5ex]{}{0_X \circ p_X \circ p_{TX}} & & & TX
	\end{tikzcd}
	\]
	is an equalizer in $\Cscr$.
\end{enumerate}
Finally, a category $\Cscr$ is said to be a tangent category if it admits a tangent structure $\Tbb$.
\end{definition}
\begin{remark}
	In \cite{GeoffRobinDiffStruct}, the axioms above are named and are given intuitive explanations. While we will not recall this here in complete detail, we will at least recall the stated content of each axiom:
	\begin{itemize}
		\item Axiom 1 names $T$ the tangent functor of the tangent structure $\Tbb$ on $\Cscr$ and $p$ the bundle map.
		\item Axiom 2 describes the object $TX \to X$ as an additive tangent bundle in $\Cscr$.
		\item Axiom 3 names $\ell$ as the vertical lift which lifts a tangent to a tangent of a tangent.
		\item  Axiom 4 names $c$ as the canonical flip which looks like the interchange of mixed partial derivatives.
		\item Axiom 5 gives the coherence relations $\ell$ and $c$ must satisfy.
		\item Axiom 6 describes the universality of the vertical lift.
	\end{itemize}
Note also that when we say that $\Cscr$ is a tangent category, we really mean that $(\Cscr,\Tbb)$ is a category equipped with a specific tangent structure $\Tbb$ and have simply left the explicit mention of the tangent structure $\Tbb$ out. In fact, this is an abuse of notation; it is possible for a category to have multiple distinct tangent structures so we will only say $\Cscr$ is a tangent category to mean that $(\Cscr,\Tbb)$ is a category equipped with a specific (potentially unspecified) tangent structure. 
\end{remark}
Equally as important to tangent categories are the morphisms of tangent categories. These come in two flavours: one, which is lax, and another which is strong. While we will usually work with strong tangent morphisms, it is important for Theorem \ref{Thm: Functoriality of ind-category} to have the full definition.
\begin{definition}[{\cite[Definition 2.7]{GeoffRobinDiffStruct}}]\label{Defn: Tangent Morphism}
Let $(\Cscr,\Tbb) = (\Cscr,T,p,0,+,\ell,c)$ and $(\Dscr,\Sbb) = (\Dscr,S,q,0^{\prime},\oplus,\ell^{\prime},c^{\prime})$ be tangent categories. A morphism of tangent categories is a pair $(F,\alpha):(\Cscr,\Tbb) \to (\Dscr,\Sbb)$ where $F:\Cscr \to \Dscr$ is a functor and $\alpha$ is a natural transformation
\[
\begin{tikzcd}
\Cscr \ar[bend left = 30, r, ""{name = U}]{}{F \circ T} \ar[r, bend right = 30, swap, ""{name = L}]{}{S \circ F} & \Dscr \ar[from = U, to = L, Rightarrow, shorten <= 4pt, shorten >= 4pt]{}{\alpha}
\end{tikzcd}
\]
for which the diagrams of functors and natural transformations
\[
\begin{tikzcd}
F \circ T \ar[r]{}{\alpha} \ar[dr, swap]{}{F \ast p} & S \circ F \ar[d]{}{q \ast F} \\
 & F
\end{tikzcd}
\begin{tikzcd}
F \ar[r]{}{F \ast 0} \ar[dr, swap]{}{0^{\prime} \ast F} & F \circ T \ar[d]{}{\alpha} \\
 & S \circ F
\end{tikzcd}
\begin{tikzcd}
F \circ T_2 \ar[r]{}{F \ast\alpha_2} \ar[d, swap]{}{F \ast +} & S_2 \circ F \ar[d]{}{\oplus \ast F} \\
F \circ T \ar[r, swap]{}{\alpha} & S \circ F
\end{tikzcd}
\]
\[
\begin{tikzcd}
F \circ T \ar[rr]{}{\alpha} \ar[d, swap]{}{F \ast \ell} & & S \circ F \ar[d]{}{\ell^{\prime} \ast F} \\
F \circ T^2 \ar[rr, swap]{}{(S \ast \alpha) \circ (\alpha \ast T)} & & S^2 \circ F
\end{tikzcd}
\begin{tikzcd}
F \circ T^2 \ar[rr]{}{(S \ast \alpha) \circ (\alpha \ast T)} \ar[d, swap]{}{F \ast c} & & S^2 \circ F \ar[d]{}{c^{\prime} \ast F} \\
F \circ T^2 \ar[rr, swap]{}{(S \ast \alpha) \circ (\alpha \ast T)} & & S^2 \circ F
\end{tikzcd}
\]
commute. Additionally, we say that the morphism $(F,\alpha)$ is strong if $\alpha$ is a natural isomorphism and if $F$ preserves the equalizers and pullbacks of the tangent structure $(\Cscr,\Tbb)$.
\end{definition}

To finish up this section we give a short description of the $1$ and $2$-categories of tangent categories, tangent morphisms, and their natural transformations.
\begin{definition}
The category $\Tan$ is defined as follows:
\begin{itemize}
	\item Objects: Tangent categories $(\Cscr, \Tbb)$.
	\item Morphisms: Tangent morphisms $(F,\alpha):(\Cscr, \Tbb) \to (\Dscr, \Sbb)$.
	\item Composition: The composition of two tangent morphisms $(F, \alpha):(\Cscr, \Tbb) \to (\Dscr, \Sbb)$ and $(G, \beta):(\Dscr, \Sbb) \to (\Escr, \Rbb)$ is defined to be the pair 
	\[
	\big(G \circ F, (\beta \ast F) \circ (G \ast \alpha)\big).
	\]
	\item Identities: The identity on a tangent category is $(\id_{\Cscr}, e):(\Cscr, \Tbb) \to (\Cscr, \Tbb)$ where $e$ is the natural transformation witnessing the equality $\id_{\Cscr} \circ T = T \circ \id_{\Cscr}$.
\end{itemize}
\end{definition}
\begin{definition}
The $2$-category $\fTan$ is defined by:
\begin{itemize}
	\item Zero-and-one-morphisms: As in $\Tan$.
	\item $2$-morphisms: A $2$-morphism as in the displayed diagram
	\[
	\begin{tikzcd}
	(\Cscr, \Tbb) \ar[rr, bend left = 30, ""{name = U}]{}{(F, \alpha)} \ar[rr, bend right = 30, swap, ""{name = L}]{}{(G, \beta)} & & (\Dscr, \Sbb) \ar[from = U, to = L, Rightarrow, shorten <= 4pt, shorten >= 4pt]{}{\rho}
	\end{tikzcd}
	\]
	is a natural transformation $\rho:F \Rightarrow G:\Cscr \to \Dscr$ for which the diagram
	\[
	\begin{tikzcd}
	F \circ T \ar[r]{}{\alpha} \ar[d, swap]{}{\rho \ast T} & S \circ F \ar[d]{}{S \ast \rho} \\
	G \circ T \ar[r, swap]{}{\beta} & S \circ G
	\end{tikzcd}
	\]
	of functors and natural transformations above commutes. Said natural transformations are called tangent transformations
	\item Vertical and horizontal composition: As in  $\fCat$.
	\item Vertical composition identities: As in $\fCat$.
\end{itemize}
\end{definition}

There are two canonical forgetful functors, $\Tan \to \Cat$ and $\fTan \to \fCat$, which are defined by $(\Cscr, \Tbb) \mapsto \Cscr$ on $0$-morphisms (objects), $(F, \alpha) \mapsto F$ on $1$-morphisms, and by $\rho \mapsto \rho$ on $2$-morphisms. Ultimately our main two theorems in this paper (cf.\@ Theorems \ref{Thm: Ind Tangent Category} and \ref{Thm: Functoriality of ind-category}) can be summaried as showing that the $\Ind$-pseudofunctor defined on $\fCat$ can be restricted/extended to a pseudofunctor $\Ind:\fTan \to \fTan$ and that the diagram
\[
\begin{tikzcd}
\fTan \ar[d, swap]{}{\Forget} \ar[r]{}{\Ind} & \fTan \ar[d]{}{\Forget} \\
\fCat \ar[r, swap]{}{\Ind} & \fCat
\end{tikzcd}
\]
commutes strictly; cf.\@ Corollary \ref{Cor: Commute strictly diagram} below.
%\begin{remark}
%While we will not need it in this paper, it is worth noting that the functor $\Forget:\fTan \to \fCat$ (and its $1$-categorical version) is surjective on objects because each category $\Cscr$ admits a tangent structure. For instance, the identity structure $\Ibb$ induced by setting the tangent functor to be the identity functor on $\Cscr$ is always a tangent structure on $\Cscr$, and in fact any tangent category of interest will almost certainly have the property that its underlying category admits multiple non-equivalent tangent structures. As an example of this phenomenon, if $\Tbb_{\Zar{S}}$ is the Zariski tangent structure on $\Sch_{/S}$ for some base scheme $S$ (cf.\@ Theorem \ref{Thm: Zariski tangent structure}) then the category $\Sch_{/S}$ has at least two tangent structures.
%\end{remark}

\section{Tangent Ind-Categories}

In this section we examine tangent structures on ind-completions of tangent categories. The general intuition here is that if we have a category of geometric objects in which we can form tangent bundles, when we take formal completions of these objects along closed subobjects we can make a formal tangent object along these completions. As an application, we will show and discuss the Zariski tangent structure on the category of formal schemes. We begin this journey with a recollection of basic properties of the ind-category.

\begin{remark}
	We have chosen to follow the more explicit and hands-on definition of tangent categories for the majority this section (as opposed to the Weil algebra approach) for use later in describing equivariant tangent ind-categories and introducing the tangent structure on equivariant formal schemes. While this does mean there is more explicit combinatorial work to be done, it has the benefit of making the theory more explicit and amenable to calculations in practice. That being said, we will review Weil algebras briefly before discussing the free tangent functor.
\end{remark}

If $\Cscr$ is a category, the ind-category $\Ind(\Cscr)$ is the free filtered formal cocompletion of $\Cscr$. The construction and presentation of $\Ind(\Cscr)$ we use is essentially the one in \cite[Section I.8.2]{SGA4}; an alternative but equivalent presentation is in \cite[Chapter 6]{KashiwaraSchapira}. Where necessary in what follows, we write $\yon:\Cscr \to [\Cscr^{\op},\Set]$ for the Yoneda embedding of $\Cscr$. We also will generally omit size-related issues in our discussion, but warn here that we should be careful to work with categories which are $\Uscr$-small for some fixed Grothendieck universe $\Uscr$ in order to avoid issues such as having large hom-sets. That being said, we will not generally worry about size-related issues in this paper.

\begin{definition}
	If $\Cscr$ is a locally small category, its presheaf-valued ind-category $\Ind_{\mathbf{PSh}}(\Cscr)$ is defined as follows:
	\begin{itemize}
		\item Objects: Presheaves $P \in [\Cscr^{\op},\Set]$ for which there is an isomorphism of functors
		\[
		P \cong \lim_{\substack{\longrightarrow \\ i \in I}} \Cscr(-,X_i) = \lim_{\substack{\longrightarrow \\ i \in I}} \yon X_i
		\]
		where $I$ is a filtered category.
		\item Morphisms: For any objects $P,Q \in \Ind(\Cscr)_0$, we define
		\[
		\Ind(\Cscr)(P,Q) := [\Cscr^{\op},\Set](P,Q).
		\]
		\item Composition and identities: As in $[\Cscr^{\op},\Set]$.
	\end{itemize}
	We will call the objects of $\Ind_{\mathbf{PSh}}(\Cscr)$ ind-presheaves of $\Cscr$.
\end{definition}
In \cite{KashiwaraSchapira} the category $\Ind_{\mathbf{PSh}}(\Cscr)$ is taken as the definition of the ind-category of $\Cscr$. We will take an approach more like the one used in \cite{SGA4} where we think of $\Ind(\Cscr)$ as the category of filtered colimits by keeping track of the filtered diagram in $\Cscr$ which defines the colimit. The only difficult part of this more ``representation agnostic'' view is that it is more difficult to define hom-sets. However, from the presheaf description we have above and the co(Yoneda Lemma) we get natural isomorphisms which tell us how to define the $\Ind$-hom sets. Fix two $\Ind$-presheaves $P$ and $Q$ over $\Cscr$ and write each as a filtered colimit
\[
P \cong \lim_{\substack{\longrightarrow \\ i \in I}}\yon(X_i)
\]
and
\[
Q \cong \lim_{\substack{\longrightarrow \\ j \in J}}\yon(X_j).
\] 
We then have by definition, the co(Yoneda Lemma), the Yoneda Lemma, and the fact that (co)limits in $[\Cscr^{\op},\Set]$ are representable (cf.\@ \cite[Equation I.8.2.5.1]{SGA4})
\begin{align*}
\Ind_{\mathbf{PSh}}(\Cscr)(P,Q) &\cong [\Cscr^{\op},\Set]\left(\lim_{\substack{\longrightarrow \\ i \in I}} \yon X_i, \lim_{\substack{\longrightarrow \\ j \in J}} \yon Y_j\right) \cong \lim_{\substack{\longleftarrow \\ i \in I}} [\Cscr^{\op},\Set]\left(\yon X_i, \lim_{\substack{\longrightarrow \\ j \in J}} \yon Y_j\right) \\
&\cong \lim_{\substack{\longleftarrow \\ i \in I}}\left(\lim_{\substack{\longrightarrow \\ j \in J}}[\Cscr^{\op},\Set](\yon X_i, \yon Y_j)\right) \cong \lim_{\substack{\longleftarrow \\ i \in I}}\left(\lim_{\substack{\longrightarrow \\ j \in J}}\Cscr(X_i,Y_j)\right).
\end{align*}
In particular, this leads us to a definition of hom-sets in a ``representation agnostic'' formulation of $\Ind(\Cscr)$.
\begin{definition}[{\cite[Section I.8.2]{SGA4}}]
	If $\Cscr$ is a category then we define an ind-object of $\Cscr$ to be a functor $F:I \to \Cscr$ where $I$ is a filtered category. If no confusion is likely to arise, we will write ind-objects as systems $\ul{X} = (X_i)_{i \in I} = (X_i)$ where $X_i := F(i)$ for all $i \in I_0$ and leave the transition morphisms $F(f), f \in I_1$ and even the functor $F$ implicit.
\end{definition}
\begin{definition}[{\cite[Section I.8.2]{SGA4}}]
	The $\Ind$-category $\Ind(\Cscr)$ of a category $\Cscr$ is defined as follows:
	\begin{itemize}
		\item Objects: Ind-objects $\ul{X}$ of $\Cscr$.
		\item Morphisms: We define the hom-sets $\Ind(\Cscr)(\ul{X},\ul{Y})$ for $\ul{X} = (X_i)_{i \in I}$ and $\ul{Y} = (Y_j)_{j \in J}$ via
		\[
		\Ind(\Cscr)(\ul{X},\ul{Y}) := \lim_{\substack{\longleftarrow \\ i \in I}}\left(\lim_{\substack{\longrightarrow \\ j \in J}}\Cscr(X_i,Y_j)\right).
		\]
		\item Composition: Induced by degree-wise composition after taking the filtered limit of the filtered colimit of maps $X_i \to Y_j$.
		\item Identities: $\id_{(X_i)}$ is the identity transformation for the functor $\ul{X}:I \to \Cscr$.
	\end{itemize}
	
\end{definition}
\begin{remark}\label{Remark: Morphisms in ind between same indexing cat}
	While the definition of hom-sets above looks complicated in general, defining morphisms between ind-objects that have the same indexing category is straightforward. In the case $\ul{X} = (X_i)_{i \in I}$ and $\ul{Y} = (Y_i)_{i \in I}$ then a morphism $\rho:\ul{X} \to \ul{Y}$ is given by a collection of maps $\rho_i:X_i \to Y_i$ for all $i \in I$ which are compatible with the transition morphisms of each functor. In particular, $\rho$ is a natural transformation
	\[
	\begin{tikzcd}
	I \ar[rr, bend left = 30, ""{name = U}]{}{F} \ar[rr, bend right = 30, swap, ""{name = L}]{}{G} & & \Cscr \ar[from = U, to = L, Rightarrow, shorten <= 4pt, shorten >= 4pt]{}{\rho}
	\end{tikzcd}
	\]
	when $F$ and $G$ are the functors representing $\ul{X}$ and $\ul{Y}$, respectively.
\end{remark}
When working with $\Ind(\Cscr)$ it is frequently helpful to represent each object as a presheaf (and in fact we will use this technique below); we will thus follow \cite{SGA4} and define the functor $L:\Ind(\Cscr) \to [\Cscr^{\op},\Set]$ as follows:
\begin{itemize}
	\item For any ind-object $\ul{X} = (X_i)_{i \in I}$ define $L(\ul{X})$ to be the presheaf
	\[
	L(\ul{X}) = \lim_{\substack{\longrightarrow \\ i \in I}}\yon X_i.
	\]
	\item For any morphism $\rho:\ul{X} \to \ul{Y}$ in $\Ind(\Cscr)$ we define $L(\rho)$ to be the image of $\rho$ under the natural isomorphism
	\[
	\Ind(\Cscr)(\ul{X},\ul{Y}) = \lim_{\substack{\longleftarrow \\ i \in I}}\left(\lim_{\substack{\longrightarrow \\ j \in J}}\Cscr(X_i,Y_j)\right) \cong [\Cscr^{\op},\Set]\left(\lim_{\substack{\longrightarrow \\ i \in I}} \yon X_i, \lim_{\substack{\longrightarrow \\ j \in J}} \yon Y_j\right) = [\Cscr^{\op},\Set](L\ul{X},L\ul{Y}).
	\]
\end{itemize}
\begin{proposition}[{\cite[Section I.8.2.4; cf.\@ Line I.8.2.4.8]{SGA4}}]\label{Prop: L functor on ind is fully faithful}
	The functor $L:\Ind(\Cscr) \to [\Cscr^{\op}, \Set]$ is fully faithful, exact, and has essential image equal to the category $\Ind_{\mathbf{PSh}}(\Cscr)$. In particular, $\Ind(\Cscr) \simeq \Ind_{\mathbf{PSh}}(\Cscr)$.
\end{proposition}
This means that when checking to see if certain objects are isomorphic in the ind-category, where constructions are often more straightforward to perform (there is less worry about making sure you haven't done something weird and misused the Yoneda Lemma), and then check for isomorphisms through the functor $L$.\footnote{Because $L$ is fully faithful, it is in particular conservative so isomorphisms $L\ul{X} \cong L\ul{Y}$ come uniquely from isomorphisms $\ul{X} \cong \ul{Y}$.} We now recall and discuss the construction of ind-functors so that we can then build ind-natural transformations and introduce the ind-pseudofunctor.

\begin{proposition}[{\cite[Section I.8.6.1]{SGA4}}]\label{Prop: Ind functors}
	If $F:\Cscr \to \Dscr$ is a functor then there is a functor $\Ind(F):\Ind(\Cscr) \to \Ind(\Dscr)$ for which given a composable pair of functors $\Cscr \xrightarrow{F} \Dscr \xrightarrow{G} \Escr$ there is a natural isomorphism
	\[
	\Ind(G \circ F) \cong \Ind(G) \circ \Ind(F).
	\]
\end{proposition}
\begin{remark}\label{Remark: Ind functor and morphisms}
The construction of the functor $\Ind(F):\Ind(\Cscr) \to \Ind(\Dscr)$ is important for our applications, so we will describe it here in the two main ways of working with the category $\Ind(\Cscr)$: the reduced $\ul{X} = (X_i)$ language and the more formal language $\ul{X}:I \to \Cscr$. In the first case, the functor $\Ind(F)(\ul{X})$ acts by
\[
\Ind(F)(\ul{X}) := (FX_i).
\]
In the second case we find that
\[
\Ind(F)(\ul{X}) = F \circ \ul{X}:I \to \Dscr.
\]
The assignment on morphisms is more complicated to describe. If we have a morphism \[
\rho \in \Ind(\Cscr)((X_i)_{i \in I}, (Y_j)_{j \in J})
\] 
then $\Ind(F)(\rho)$ is induced by the assignment
\[
\lim_{I}\left(\colim_{J}\rho_{ij}\right) \mapsto \lim_{I}\left(\colim_{J}F(\rho_{ij})\right)
\]
However, there is a clean description in case $\ul{X}, \ul{Y}:I \to \Cscr$ are functors defined on the same indexing category. In this case a morphism $\rho:\ul{X} \to \ul{Y}$ is exactly a natural transformation
\[
\begin{tikzcd}
I \ar[rr, bend left = 30, ""{name = Up}]{}{\ul{X}} \ar[rr, bend right =30, ""{name = Down}, swap]{}{\ul{Y}} & & \Cscr \ar[from = Up, to = Down, Rightarrow, shorten <= 4pt, shorten >= 4pt]{}{\rho}
\end{tikzcd}
\]
and so we define $\Ind(F)$ algebraically by
\[
\Ind(F)(\rho_i) = (F\rho_i)
\]
and, in the categorical perspective,
\[
\Ind(F)(\rho) = F \ast \rho.
\]
\end{remark}
\begin{corollary}\label{Cor: Ind is rigidified}
For any category $\Cscr$ there is an equality $\Ind(\id_{\Cscr}) = \id_{\Ind(\Cscr)}$.
\end{corollary}
\begin{proof}
The verification that $\id_{\Ind(\Cscr)}(\ul{X}) = \ul{X} = \Ind(\id_{\Cscr})(\ul{X})$ for any object $\ul{X}$ is trivial and omitted. For the case of morphisms note that if $\lim(\colim \rho_{ij}):(X_i) \to (Y_j)$ is any morphism then $\Ind(\id_{\Cscr})$ acts on $\lim(\colim \rho_{ij})$ via
\[
\lim(\colim \rho_{ij}) \mapsto \lim\bigg(\colim\big(\id_{\Cscr}(\rho_{ij})\big)\bigg) = \lim(\colim(\rho_{ij})),
\]
which is exactly the identity assignment.
\end{proof}
\begin{remark}
As we proceed in this paper, for any composable pair of functors $\Cscr \xrightarrow{F} \Dscr \xrightarrow{G} \Escr$ we write $\phi_{F,G}$ for the compositor natural isomorphism:
\[
\begin{tikzcd}
\Ind(\Cscr) \ar[rr, bend left = 30, ""{name = U}]{}{\Ind(G) \circ \Ind(F)} \ar[rr, swap, bend right = 30, ""{name = L}]{}{\Ind(G \circ F)} & & \Ind(\Escr) \ar[from = U, to = L, Rightarrow, shorten <= 4pt, shorten >= 4pt]{}{\phi_{F,G}} \ar[from = U, to = L, Rightarrow, shorten <= 4pt, shorten >= 4pt, swap]{}{\cong}
\end{tikzcd}
\]
\end{remark}

While in general we do only have a natural isomorphism $\Ind(G \circ F) \cong \Ind(G) \circ \Ind(F)$, the lemma we prove below (for use later when we prove that $\hat{\ell}$ and $\hat{c}$ are components of bundle maps) shows that these compositor natural isomorphisms are actually relatively well-behaved in the sense that the action of $\Ind(G) \circ \Ind(F)$ and $\Ind(G \circ F)$ on certain classes of morphisms are the same.
\begin{lemma}\label{Lemma: Ind of composition on morphism from same index category agree}
	Let $F:\Cscr \to \Dscr$ and $G:\Dscr \to \Escr$ be functors. If $\ul{X}, \ul{Y}:I \to \Cscr$ are two objects in $\Ind(\Cscr)$ with ${\rho} = (\rho_i)_{i \in I}:\ul{X} \to \ul{Y}$ an ind-morphism, then
	\[
	\big(\Ind(G) \circ \Ind(F)\big)(\ul{\rho}) = \Ind(G \circ F)(\ul{\rho}).
	\]
\end{lemma}
\begin{proof}
	Because both $\ul{X}$ and $\ul{Y}$ have the same indexing category, it follows from Remark \ref{Remark: Ind functor and morphisms}	
	\[
	\Ind(G \circ F)(\ul{\rho}) = (G \circ F) \ast \rho
	\]
	Similarly, we find that
	\[
	\Ind(F)({\rho}) = F \ast \rho = \iota_F \ast \rho
	\]
	where $\iota_F$ is the identity natural transformation on $F$. Then
	\[
	\Ind(G)\big(F\rho_i\big)_{i \in I} = G \ast (\iota_F\ast \rho) = \iota_G\ast(\iota_F \ast \rho) = (\iota_G\ast\iota_F) \ast \rho = (\iota_{G \circ F}) \ast \rho.
	\]
	Thus we conclude that
	\[
	\Ind(G \circ F)(\rho_i) = (G \circ F) \ast \rho= \iota_{G\circ F} \ast \rho = \iota_G \ast (\iota_F \ast \rho) = \Ind(G)\big(\Ind(F)(\rho)\big),
	\]
	as was desired.
\end{proof}

\begin{proposition}\label{Prop: Ind nat transform induced by nat transform}
	Let $\Cscr, \Dscr$ be categories with functors $F,G:\Cscr \to \Dscr$ and let $\alpha:F \Rightarrow G$ be a natural transformation. Then there is a natural transformation
	\[
	\Ind(\alpha):\Ind(F) \Rightarrow \Ind(G).
	\]
\end{proposition}
\begin{proof}
	Fix $\ul{X} = (X_i) \in \Ind(\Cscr)_0$. We define $\Ind(\alpha)_{\ul{X}}:\Ind(F)(X_i) \to \Ind(G)(X_i)$ by setting (cf.\@ Remark \ref{Remark: Morphisms in ind between same indexing cat})
	\[
	\Ind(\alpha)_{\ul{X}} := (\alpha_{X_i})_{i \in I}.
	\]
	To check that this is a natural transformation fix a map $\rho:(X_i)_{i \in I} \to (Y_j)_{j \in J}$ so that the sequences $(\rho_{ij})$ are a family of morphisms compatible with the transition morphisms of $\ul{X}$ and $\ul{Y}$. In particular, for any $f:i \to i^{\prime}$ in $I$ and $g:j \to j^{\prime}$ in $J$ the square
	\[
	\xymatrix{
		X_i \ar[r]^-{\rho_{ij}} \ar[d]_{f_{\ast}} & Y_j \ar[d]^{g_{\ast}} \\
		X_{i^{\prime}} \ar[r]_-{\rho_{i^{\prime}j^{\prime}}} & Y_{j^{\prime}}
	}
	\] 
	commutes. To prove naturality, it suffices to show that for any any $f:i \to i^{\prime}$ in $I$ and for any $g:j \to j^{\prime}$ in $J$, the cube
	\[
	\begin{tikzcd}
	FX_i \ar[dr, swap]{}{F\rho_{ij}} \ar[rr]{}{\alpha_{X_i}} \ar[dd, swap]{}{F(f_{\ast})} & & GX_i  \ar[dd, near start]{}{G(f_{\ast})} \ar[dr]{}{G\rho_{ij}} \\
	& FY_j  & & GY_j \ar[dd]{}{G(g_{\ast})} \\
	FX_{i^{\prime}} \ar[dr, swap]{}{F(\rho_{i^{\prime}j^{\prime}})} \ar[rr, near start]{}{\alpha_{X_{i^{\prime}}}} & & GX_{i^{\prime}} \ar[dr, near start]{}{G\rho_{i^{\prime}j^{\prime}}} \\
	& FY_{j^{\prime}} \ar[rr, swap]{}{\alpha_{Y_{j^{\prime}}}} & & GY_{j^{\prime}}
	\ar[from = 2-2, to = 2-4, crossing over, near start]{}{\alpha_{Y_j}} \ar[from = 2-2, to = 4-2, crossing over, near start]{}{F(g_{\ast})}	
	\end{tikzcd}
	\]
	commutes; note that we are writing $f_{\ast}$ for the image of $f:i \to i^{\prime}$ as a transition morphism of $\ul{X}$ and similarly for $g_{\ast}$ as  a transtion morphism of $\ul{Y}$. However, note that the squares
	\[
	\begin{tikzcd}
	FX_i \ar[r]{}{\alpha_{X_i}} \ar[d, swap]{}{F(f_{\ast})} & GX_i \ar[d]{}{G(f_{\ast})} \\
	FX_{i^{\prime}} \ar[r, swap]{}{\alpha_{X_{i^{\prime}}}} & GX_{i^{\prime}}
	\end{tikzcd}\quad
	\begin{tikzcd}
	FY_j \ar[r]{}{\alpha_{Y_j}} \ar[d, swap]{}{F(g_{\ast})} & GY_j \ar[d]{}{G(g_{\ast})} \\
	FY_{j^{\prime}} \ar[r, swap]{}{\alpha_{Y_j}} & GY_{j^{\prime}}
	\end{tikzcd}
	\]
	commute by the naturality of $\alpha$. Similarly, the squares
	\[
	\begin{tikzcd}
	FX_i \ar[r]{}{F\rho_{ij}} \ar[d, swap]{}{\alpha_{X_i}} & FY_j \ar[d]{}{\alpha_{Y_j}} \\
	GX_{i} \ar[r, swap]{}{G\rho_{ij}} & GY_j
	\end{tikzcd}\quad
	\begin{tikzcd}
	FX_{i^{\prime}} \ar[r]{}{F(\rho_{i^{\prime}j^{\prime}})} \ar[d, swap]{}{\alpha_{X_{i^{\prime}}}} & FY_{j^{\prime}} \ar[d]{}{\alpha_{Y_j^{\prime}}} \\
	GX_{i^{\prime}}  \ar[r, swap]{}{G(\rho_{i^{\prime}j^{\prime}})} & GY_{j^{\prime}}
	\end{tikzcd}
	\]
	commute by assumption of the $\rho_{ij}$ being compatible with the transition morphisms and the functoriality of $F$ and $G$. Thus the cube commutes, which shows that $\Ind(\alpha)$ is indeed natural after taking the filtered limit of the filtered colimit of maps.
\end{proof}
\begin{corollary}\label{Cor: Ind is strict 2-mor}
Given a two-cell of categories
\[
\begin{tikzcd}
\Cscr \ar[rr, bend left = 50, ""{name = U}]{}{F} \ar[rr, ""{name = M}]{}[description]{G} \ar[rr, bend right = 50, swap, ""{name = L}]{}{H} & & \Dscr \ar[from = U, to = M, Rightarrow, shorten <= 6pt]{}{\alpha} \ar[from = M, to = L, Rightarrow, shorten <= 6pt, shorten >= 4pt]{}{\beta}
\end{tikzcd}
\]
we have
\[
\Ind(\beta \circ \alpha) = \Ind(\beta) \circ \Ind(\alpha).
\]
\end{corollary}
\begin{proof}
Let $\ul{X} = (X_i)_{i \in I}$ be an object of $\Ind(\Cscr)_0$. We then calculate from Proposition \ref{Prop: Ind nat transform induced by nat transform} that
\[
\Ind(\beta \circ \alpha)_{\ul{X}} = \left((\beta \circ \alpha)_{X_i}\right)_{i \in I} = (\beta_{X_i} \circ \alpha_{X_i})_{i \in I} = (\beta_{X_i})_{i \in I} \circ (\alpha_{X_i})_{i \in I} = \Ind(\beta)_{\ul{X}} \circ \Ind(\alpha)_{\ul{X}}.
\]
\end{proof}

We now need to know how the $\Ind$-assignment interacts with horizontal composition of natural transformations in order to verify that it is pseudofunctorial in $\fCat$. Recall that pseudofuctoriality declares immediately that if we have a $2$-cell
\[
\begin{tikzcd}
\Cscr \ar[rr, bend left = 30, ""{name = UL}]{}{F} \ar[rr, swap, bend right = 30, ""{name = LL}]{}{G}  & & \Dscr \ar[from = UL, to = LL, Rightarrow, shorten <= 4pt, shorten >= 4pt]{}{\alpha} \ar[rr, bend left = 30, ""{name = UR}]{}{H} \ar[rr, bend right = 30, swap, ""{name = LR}]{}{K} & & \Escr \ar[from = UR, to = LR, Rightarrow, shorten <= 4pt, shorten >= 4pt]{}{\beta}
\end{tikzcd}
\]
then we have the identity
\[
\Ind(\beta) \ast \Ind(\alpha) = \phi_{G,K}^{-1} \circ \Ind(\beta \ast \alpha) \circ \phi_{F,H};
\]
in particular, establishing the above identity is equivalent to establishing pseudofunctoriality on $2$-morphisms.
\begin{lemma}
The $\Ind$-assignment is pseudofunctorial on $2$-morphisms in the sense that if we have natural transformations $\alpha:F \Rightarrow G:\Cscr \to \Dscr$ and if $\beta:H \Rightarrow K:\Dscr \to \Escr$ then 
\[
\Ind(\beta \ast \alpha) = \phi_{G,K}^{-1} \circ \Ind(\beta) \ast \Ind(\alpha) \circ \phi_{F,H}.
\]
\end{lemma}
\begin{proof}
This is a straightforward verifiction using Lemma \ref{Lemma: Ind of composition on morphism from same index category agree} and the observation that every object and morphism in sight in the definitions of $\Ind(\beta \ast \alpha)$ and $\Ind(\beta) \ast \Ind(\alpha)$ involve only the same indexing categories.
\end{proof}

%However, we would like to see if we can sharpen this in a similar way to Corollary \ref{Cor: Ind is strict 2-mor}. Note that if $\ul{X} = (X_i)_{i \in I}$ is any object in $\Ind(\Cscr)$ then we have
%\[
%\big(\Ind(H) \circ \Ind(F)\big)\ul{X} = \Ind(H)\big(FX_i\big)_{i \in I} = \big(H(FX_i)\big)_{i \in I} = \big((H \circ F)(X_i)\big)_{i \in I} = \Ind(H \circ F)\ul{X}
%\]
%and similarly
%\[
%\Ind(K \circ G)\ul{Y} = \big(\Ind(K) \circ \Ind(G)\big)\ul{Y}
%\]
%for any object $\ul{Y} = (Y_j)_{j \in J}$ in $\Ind(\Dscr)$. This allows us to deduce the equations
%\begin{equation}
%\Ind(\beta \ast \alpha)_{\ul{X}} = \big((\beta \ast \alpha)_{X_i}\big)_{i \in I}
%\end{equation}
%and
%\begin{equation}
%\left(\Ind(\beta) \ast \Ind(\alpha)\right)_{\ul{X}} = \left(\phi_{G,K}^{-1}\right)_{\ul{X}} \circ \left((\beta \ast \alpha)_{X_i}\right)_{i \in I} \circ \left(\phi_{F,K}\right)_{\ul{X}}.
%\end{equation}
%We can also check that the 

%\begin{lemma}
%If $\alpha:F \Rightarrow G:\Cscr \to \Dscr$ and if $\beta:H \Rightarrow K:\Dscr \to \Escr$ are natural transformations then $\Ind(\beta \ast \alpha) = \Ind(\beta) \ast \Ind(\alpha)$.
%\end{lemma}
%\begin{proof}
%Let $\ul{X} = (X_i)_{i \in I} \in \Ind(\Cscr)_0$. Then 
%\[
%asdf
%\]
%\end{proof}
%\begin{proposition}
%The pseudofunctor $\Ind:\fCat \to \fCat$ is strict on horizontal and vertical composition of $2$-morphisms.
%\end{proposition}
%\begin{proof}
%
%\end{proof}

We now will endeavour to show that if $(\Cscr,\Tbb)$ is a tangent category then the ind-category $\Ind(\Cscr)$ naturally inherits an ind-tangent structure which we will call $\Ind(\Tbb)$ from $\Tbb$. We begin this study by first showing the existence of the natural transformations $\Ind(p)$ and $\Ind(0)$; constructing these is straightforward, while the natural transformations $\hat{+}, \hat{\ell}$, and $\hat{c}$ take a little more work to define for technical $2$-categorical reasons: the $2$-functor $\Ind:\fCat \to \fCat$ is a pseudofunctor and not strictly functorial.

\begin{lemma}\label{Lemma: Existence of Indp and Ind0}
	If $\Cscr$ is a tangent category then there are ind-bundle and ind-zero natural transformations $\Ind(p):\Ind(T) \Rightarrow \id_{\Ind(\Cscr)}$ and $\Ind(0):\id_{\Ind(\Cscr)} \Rightarrow \Ind(T)$.
\end{lemma}
\begin{proof}
	Apply Corollary \ref{Cor: Ind is rigidified} and Proposition \ref{Prop: Ind nat transform induced by nat transform} to the natural transformations $p:T \Rightarrow \id_{\Cscr}$ and $0:\id_{\Cscr} \Rightarrow T$.
\end{proof}

Because every tangent category $\Cscr$ admits (finite) pullbacks against the bundle maps $p_X:TX \to X$ for every object $X \in \Cscr_0$, it follows from \cite[Proposition I.8.9.1.c]{SGA4} that $\Ind(\Cscr)$ admits finite pullbacks agains the ind-bundle maps $\Ind(p)_{\ul{X}}:\Ind(T)\ul{X} \to \ul{X}$ for every $\ul{X} \in \Ind(\Cscr)_0$. In order to have our tangent structure $\Ind(\Tbb)$ be essentially given as the indicization of the tangent structure $\Tbb$, we nee now must establish a natural isomorphism (and hence also the existence of) $\Ind(T) {}_{\Ind(p)}\times_{\Ind(p)} \Ind(T) = \Ind(T)_2 \cong \Ind(T_2)$. Afterwards, we'll establish that compositional powers of $\Ind(T)$ preserve and commute with these pullback powers before constructiong the ind-bundle addition and proving that $\Ind(p)_{\ul{X}}$ is a commutative monoid in $\Ind(\Cscr)_{/\ul{X}}$.

\begin{proposition}\label{Prop: Ind of Tangent pullback is tangent pullback of ind}
	If $(\Cscr,\Tbb)$ is a tangent category then there is a natural isomorphism of functors $\Ind(T_2) \cong \Ind(T)_2:\Ind(\Cscr) \to \Ind(\Cscr)$.
\end{proposition}
\begin{proof}
	First that the pullback powers $\Ind(T_n)\ul{X} = \Ind(T)\ul{X} \times_{\ul{X}} \cdots \times_{\ul{X}} \Ind(T)\ul{X}$ follows from the fact that for any objects $\ul{X} = (X_i), \ul{Y} = (Y_i), \ul{Z} = (Z_i)$ the object $\ul{W} = (X_i \times_{Z_i} Y_i)$ represents the pullback $\ul{X} \times_{\ul{Z}} \ul{Y}$ in $\Ind(\Cscr)$; the verification of this fact is routine and straightforward.
	
	We use the presheaf realization functor $L:\Ind(\Cscr) \to [\Cscr^{\op},\Set]$ to prove the remainder of proposition, i.e., that $\Ind(T)_2 \cong \Ind(T_2)$.  On objects we have that on one hand
	\[
	\Ind(T)_2\ul{X} = \Ind(T)\ul{X} \times_{\ul{X}} \Ind(T)\ul{X} = (TX_i) \times_{(X_i)} (TX_i)
	\]
	while on the other hand
	\[
	\Ind(T_2)\ul{X} = (T_2X_i) = (TX_i \times_{X_i} TX_i).
	\]
	Taking the image under $L$ and using that filtered colimits commute with finite limits in $[\Cscr^{\op},\Set]$ we find
	\begin{align*}
	L\big(\Ind(T_2)\ul{X}\big) &= \lim_{\substack{\longrightarrow \\ i \in I}}\yon(TX_i \times_{X_i} TX_i) \cong \lim_{\substack{\longrightarrow \\ i \in I}}\left( \yon TX_i \times_{\yon X_i} \yon TX_i\right) \cong \lim_{\substack{\longrightarrow \\ i \in I}}\yon TX_i \times_{\colim \yon X_i} \lim_{\substack{\longrightarrow \\ i \in I}} \yon TX_i \\
	&\cong L(TX_i) \times_{LX_i} L(TX_i) \cong L\left(\Ind(T)\ul{X} \times_{\ul{X}} \Ind(T)\ul{X}\right) = L\left(\Ind(T)_2\ul{X}\right).
	\end{align*}
	From the fact that $L$ is conservative, we conclude that $\Ind(T)_2\ul{X} \cong \Ind(T_2)\ul{X}$; that this isomorphism is natural in $\ul{X}$ because it is induced by the universal property of the limit together with limit preservation isomorphisms.
\end{proof}
\begin{corollary}\label{Cor: Ind Tangent pullbacks are tangent ind pullbacks for any m}
	For any tangent category $(\Cscr,\Tbb)$ and any $m \in \N$ there is an isomorphism $\Ind(T)_m  \cong \Ind(T_m)$.
\end{corollary}
\begin{proof}
	This follows mutatis mutandis to the proof of Proposition \ref{Prop: Ind of Tangent pullback is tangent pullback of ind}.
\end{proof}
\begin{corollary}\label{Cor: Ind Tangent functor commutes with Ind tangent pullbacks}
	Let $(\Cscr,\Tbb)$ be a tangent category and let $m, n \in \N$. Then there is a natural isomorphism of functors
	\[
	\Ind(T)^n \circ \Ind(T)_m \cong \Ind(T)_m \circ \Ind(T)^n.
	\]
\end{corollary}
\begin{proof}
	Since $\Cscr$ is a tangent category, we have natural a natural isomorphism $T^n \circ T_m \cong T_m \circ T^n$. Apply Proposition \ref{Prop: Ind nat transform induced by nat transform} to this natural isomorphism to get a natural isomorphism $\Ind(T^n \circ T_m) \cong \Ind(T_m \circ T^n)$. Finally, conjugating by the to the natural isomorphism by the $\Ind$-compositor $\phi_{f,g}:\Ind(f) \circ \Ind(g) \to \Ind(g \circ f)$ multiple times together together with using Corollary \ref{Cor: Ind Tangent pullbacks are tangent ind pullbacks for any m} gives
	\begin{align*}
	\Ind(T)^n \circ \Ind(T)_m &\cong \Ind(T^n) \circ \Ind(T_m) \cong  \Ind(T^n \circ T_m) \cong \Ind(T_m \circ T^n) \cong \Ind(T_m) \circ \Ind(T^n) \\
	&\cong \Ind(T)_m \circ \Ind(T)^n.
	\end{align*}
\end{proof}

Let us now prove the existence of the ind-addition natural transformation $\Ind(+)$ and prove that for any $\ul{X} \in \Ind(\Cscr)_0$, $\Ind(p)_{\ul{X}}$ is a commutative monoid in $\Ind(\Cscr)_{/\ul{X}}$.
\begin{lemma}\label{Lemma: Existence of ind-bundle addition}
	Let $(\Cscr,\Tbb)$ be a tangent category. Then there is an ind-addition natural transformation
	\[
	\hat(+):\Ind(T)_2 \Rightarrow \Ind(T).
	\]
\end{lemma}
\begin{proof}
	We define $\Ind(+)$ to be the natural transformation displayed below
	\[
	\xymatrix{
		\Ind(T)_2 \ar[r]^-{\cong} & \Ind(T_2) \ar[r]^-{\Ind(+)} & \Ind(T)
	}
	\]
	where $\Ind(+)$ is the indicization of the natural transformation $+:T_2 \Rightarrow T$ of Proposition \ref{Prop: Ind nat transform induced by nat transform}.
\end{proof}
\begin{proposition}\label{Prop: Additive Ind Bundle}
	For any tangent category $(\Cscr,\Tbb)$ and any $\ul{X} \in \Ind(\Cscr)_0$, the object $\hat(p)_{\ul{X}}:\Ind(T)\ul{X} \to \ul{X}$ is a commutative monoid in $\Ind(\Cscr)_{/\ul{X}}$ with unit $\Ind(0)_{\ul{X}}$ and addition $\hat{+}$.
\end{proposition}
\begin{proof}
	We must show that in $\Ind(\Cscr)_{/\ul{X}}$ the diagrams
	\[
	\xymatrix{
		\Ind(T)\ul{X} \ar[r]^-{\cong} \ar@{=}[drrr] \ar[d]_{\cong} & \Ind(T)\ul{X} \times_{\ul{X}} \ul{X} \ar[rr]^-{\id_{\Ind(T)\ul{X}} \times \Ind(0)_{\ul{X}}} & & \Ind(T)_2\ul{X}\ar[d]^{\hat{+}_{\ul{X}}} \\
		\ul{X} \times_{\ul{X}} \Ind(T)\ul{X} \ar[rr]_-{\Ind(0)_{\ul{X}} \times \id_{\Ind(T)\ul{X}}} & & \Ind(T)_2\ul{X} \ar[r]_-{\hat{+}_{\ul{X}}} & \Ind(T)\ul{X}
	}
	\]
	\[
	\xymatrix{
		\Ind(T)_2\ul{X} \ar[r]^-{s} \ar[dr]_{\hat{+}_{\ul{X}}} & \Ind(T)_2\ul{X} \ar[d]^{\hat{+}_{\ul{X}}} \\
		& \Ind(T)\ul{X}
	}
	\]
	\[
	\begin{tikzcd}
	(\Ind(T)_2\ul{X}) \times_{\ul{X}} \Ind(T)\ul{X} \ar[rr]{}{\hat{+}_{\ul{X}} \times \id_{\Ind(T)\ul{X}}} \ar[d, swap]{}{\cong} & & \Ind(T)_2\ul{X} \ar[d]{}{\hat{+}_{\ul{X}}} \\
	\Ind(T)\ul{X} \times (\Ind(T)_2\ul{X}) \ar[dr, swap]{}{\id_{\Ind(T)\ul{X}} \times \hat{+}_{\ul{X}}} & & \Ind(T)\ul{X} \\
	& \Ind(T)_2\ul{X} \ar[ur, swap]{}{\hat{+}_{\ul{X}}}
	\end{tikzcd}
	\]
	commute. To prove the commutativity of these diagrams we note that $\Ind(T)_2\ul{X}$ carries the natural isomorphism $\Ind(T)_2 \cong \Ind(T_2)$ which mediates between the pullback in $\Ind(\Cscr)$ and the ind-object given by the pullbacks $(TX_i \times_{X_i} TX_i)$ and the map $\hat{+}_{\ul{X}}$ is defined by first using this mediating isomorphism before acting on the ind-object $(TX_i \times_{X_i} TX_i)$. As such, it follows by naturality to verify each of the diagrams above on the corresponding incarnation of ind-objects whose components are all given $i$-locally as the pullback of objects in $\Cscr$. More explicitly, to verify the first diagram it suffices to show that the diagram
	\[
	\xymatrix{
		(TX_i) \ar[r]^-{\cong} \ar@{=}[drrr] \ar[d]_{\cong} & (TX_i \times_{X_i} X_i) \ar[rr]^-{\id_{TX_i} \times 0_{X_i}} & & (T_2X_i) \ar[d]^{+_{X_i}} \\
		(X_i \times_{X_i} TX_i) \ar[rr]_-{0_{X_i} \times \id_{TX_i}} & & (T_2X_i) \ar[r]_-{+_{X_i}} & (TX_i)
	}
	\]
	commutes in $\Ind(\Cscr)$. However, because $TX_i$ is a bundle over $X_i$ for all $i \in I$, so the diagram
	\[
	\xymatrix{
		TX_i \ar[r]^-{\cong} \ar@{=}[drrr] \ar[d]_{\cong} & TX_i \times_{X_i} X_i \ar[rr]^-{\id_{TX_i} \times 0_{X_i}} & & T_2X_i \ar[d]^{+_{X_i}} \\
		X_i \times_{X_i} TX_i \ar[rr]_-{0_{X_i} \times \id_{TX_i}} & & T_2X_i \ar[r]_-{+_{X_i}} & TX_i
	}
	\]
	for each $i \in I$. Thus by taking the image of the diagram under the functor $L$ it follows that the diagram of ind-presheaves commutes and so via the fact that $L$ is fully faithful, it follows that the diagram commutes in $\Ind(\Cscr)$ as well. The commutativity of the other diagrams is verified mutatis mutandis to this diagrams, and so are omitted.
	%Note that since the functor $L:\Ind(\Cscr) \to [\Cscr^{\op},\Set]$ is fully faitful and exact, the induced slice functor $L_{/\ul{X}}:\Ind(\Cscr)_{/\ul{X}} \to [\Cscr^{\op},\Set]_{/L\ul{X}}$ is fully faithful and exact as well. As such, it suffices to prove that the relevant diagrams describing a commutative monoid 
\end{proof}

We now build the ind-vertical lift $\Ind(\ell):\Ind(T) \Rightarrow \Ind(T)^2$ and prove that it induces a bundle morphism in $\Ind(\Cscr)$.
\begin{lemma}\label{Lemma: Existence of Ind-vertical lift}
	If $(\Cscr, \Tbb)$ is a tangent category then there is an ind-vertical lift transformation $\hat{\ell}:\Ind(T) \Rightarrow \Ind(T)^2$.
\end{lemma}
\begin{proof}
	We define $\hat{\ell}$ as in the diagram of functors and natural transformations
	\[
	\xymatrix{
		\Ind(T) \ar[r]^-{\Ind{\ell}} & \Ind(T^2) \ar[r]^-{\phi_{T,T}} & \Ind(T)^2
	}
	\]
	where $\Ind{\ell}:\Ind(T) \Rightarrow \Ind(T^2)$ is the transformation induced by applying Proposition \ref{Prop: Ind nat transform induced by nat transform} and $\phi_{T,T}:\Ind(T^2) \cong \Ind(T)^2$ is the compositor isomorphism. 
\end{proof}
\begin{proposition}\label{Prop: Ind vertical lift is part of a bundle map}
	If $(\Cscr,\Tbb)$ is a tangent category then for any object $\ul{X}$ of $\Ind(\Cscr)$, the pair of morphisms $(\hat{\ell}_{\ul{X}},\Ind(0)_{\ul{X}})$ describes a morphism of bundles in $\Ind(\Cscr)$.
\end{proposition}
\begin{proof}
	We must prove that the diagrams
	\[
	\begin{tikzcd}
	\Ind(T)\ul{X} \ar[rr]{}{\hat{\ell}_{\ul{X}}} \ar[d, swap]{}{\Ind(p)_{\ul{X}}} & & \Ind(T)^2\ul{X} \ar[d]{}{(\Ind(T) \ast \Ind(p))_{\ul{X}}} \\
	\ul{X} \ar[rr, swap]{}{\Ind(0)_{\ul{X}}} & & \Ind(T)\ul{X}
	\end{tikzcd}
	\]
	\[
	\begin{tikzcd}
	\Ind(T)\ul{X} \times_{\ul{X}} \Ind(T)\ul{X} \ar[rrrr]{}{\langle \hat{\ell}_{\ul{X}} \circ \pi_1, \hat{\ell}_{\ul{X}} \circ \pi_2\rangle} \ar[d, swap]{}{\hat{+}_{\ul{X}}} & & & &\Ind(T)^2\ul{X} \times_{\Ind(T)\ul{X}} \Ind(T)^2\ul{X} \ar[d]{}{(\hat{+} \ast \Ind(T))_{\ul{X}}} \\
	\Ind(T)\ul{X} \ar[rrrr, swap]{}{\hat{\ell}_{\ul{X}}} & & & & \Ind(T)^2\ul{X}
	\end{tikzcd}
	\]
	\[
	\begin{tikzcd}
	\ul{X} \ar[rr]{}{\Ind(0)_{\ul{X}}} \ar[d, swap]{}{\Ind(0)_{\ul{X}}} & & \Ind(T)\ul{X} \ar[d]{}{(\Ind(0) \ast \Ind(T))_{\ul{X}}} \\
	\Ind(T)\ul{X} \ar[rr, swap]{}{\hat{\ell}_{\ul{X}}} & & \Ind(T)^2\ul{X}
	\end{tikzcd}
	\]
	commute in $\Ind(\Cscr)$. Write $\ul{X} = (X_i)_{i \in I}$ and consider that the first diagram has bottom edge calculated by
	\[
	\Ind(0)_{\ul{X}} \circ \Ind(p)_{\ul{X}} = (0_{X_i}) \circ (p_{X_i}) = (0_{X_i} \circ p_{X_i}) = ((T \ast p)_{X_i} \circ \ell_{X_i})
	\]
	because $\Cscr$ is a tangent category. Alternatively, by Lemma \ref{Lemma: Existence of Ind-vertical lift} we find that the upper half of the diagram is calculated by
	\[
	\big(\Ind(T) \ast \Ind(p)\big)_{\ul{X}} \circ \hat{\ell}_{X_i} = \big((T \ast p)_{X_i}\big) \circ (\ell_{X_i}) = \big((T \ast p)_{X_i} \circ \ell_{X_i}\big) = \Ind(0)_{\ul{X}} \circ \Ind(p)_{\ul{X}},
	\]
	so the first diagram indeed commutes. The third diagram is verified to commute similarly, so it suffices to show now that the second diagram commutes. For this we note that on one hand
	\[
	\hat{\ell}_{\ul{X}} \circ \hat{+}_{\ul{X}} = (\ell_{X_i})_{i \in I} \circ (0_{X_i})_{i \in I} = (\ell_{X_i} \circ 0_{X_i})_{i \in I} = \big((+ \ast T)_{X_i} \circ \langle \ell_{X_i} \circ \pi_{1,i}, \ell_{X_i}\circ \pi_{2,i}\rangle\big)_{i \in I}
	\]
	because $\Cscr$ is a tangent category. Now on the other hand we note that
	\begin{align*}
	\big(\hat{+} \ast \Ind(T)\big)_{\ul{X}} \circ \left\langle \hat{\ell}_{\ul{X}} \circ \pi_1, \hat{\ell}_{\ul{X}} \circ \pi_2 \right\rangle &= \left((+ \ast T)_{X_i}\right) \circ \left\langle (\ell_{X_i}) \circ (\pi_{1,i}), \ell_{X_i} \circ (\pi_{2,i})\right\rangle \\
	& = \big((+ \ast T)_{X_i}\big) \circ \left\langle (\ell_{X_i} \circ \pi_{1,i}), (\ell_{X_i} \circ \pi_{2,i}) \right\rangle \\
	&= \big((+ \ast T)_{X_i}\big) \circ \big(\langle \ell_{X_i} \circ \pi_{1,i}, \ell_{X_i} \circ \pi_{2,i} \rangle\big) \\
	&= \left((+ \ast T)_{X_i} \circ \langle \ell_{X_i}\circ \pi_{1,i}, \ell_{X_i} \circ \pi_{2,i}\rangle\right)
	\end{align*}
	so it follows that the diagram indeed commutes.
\end{proof}

We now provide the existence of the ind-canonical flip. This amounts to being able to commute partial derivative operators within formal $2$-jets of our tangent object $\Ind(T)\ul{X} \to \ul{X}$. Afterwards we will prove that $\hat{c}$ is one of the components of a bundle morphism.
\begin{lemma}\label{Lemma: Existence of ind-canonical flip}
	Let $(\Cscr,\Tbb)$ be a tangent category. Then there is an ind-canonical flip natural transformation
	\[
	\hat{c}:\Ind(T)^2 \Rightarrow \Ind(T)^2.
	\]
\end{lemma}
\begin{proof}
	We define the ind-canonical flip via the diagram
	\[
	\begin{tikzcd}
	\Ind(T)^2 \ar[d, swap, Rightarrow]{}{\hat{c}} \ar[r, Rightarrow]{}{\phi_{T,T}} & \Ind(T^2) \ar[d, Rightarrow]{}{\Ind{c}} \\
	\Ind(T)^2 & \Ind(T^2) \ar[l, Rightarrow]{}{\phi_{T,T}^{-1}}
	\end{tikzcd}
	\]
	where $\phi_{T,T}$ is the compositor isomorphism $\theta:\Ind(T)^2 \xRightarrow{\cong}\Ind(T^2)$ and $\Ind{c}$ is the indicization of the canonical flip $c:T^2 \Rightarrow T^2$ asserted by Proposition \ref{Prop: Ind nat transform induced by nat transform}.
\end{proof}
\begin{proposition}\label{Prop: Ind-canonical flip is bundle map}
	If $(\Cscr, \Tbb)$ is a tangent category then for any object $\ul{X}$ of $\Ind(\Cscr)$, the pair of morphisms $(\id_{\Ind(T)\ul{X}}, \hat{c}_{\ul{X}})$ describe a bundle morphism.
\end{proposition}
\begin{proof}
	Following Definition \ref{Defn: Tangent Categpry}, we must show that the diagrams
	\[
	\begin{tikzcd}
	\Ind(T)^2{\ul{X}} \ar[rr]{}{\hat{c}_{\ul{X}}} \ar[d, swap]{}{(\Ind(T) \ast \Ind(p))_{\ul{X}}} & & \Ind(T)^2{\ul{X}} \ar[d]{}{(\Ind(p) \ast \Ind(T))_{{\ul{X}}}} \\
	\Ind(T){\ul{X}} \ar[rr, equals, swap]{}{} & & \Ind(T){\ul{X}} 
	\end{tikzcd}
	\]
	\[
	\begin{tikzcd}
	\Ind(T)^2{\ul{X}} \times_{\Ind(T){\ul{X}}} \Ind(T)^2{\ul{X}} \ar[d, swap]{}{(\Ind(T) \ast \hat{+})_{{\ul{X}}}} \ar[rrrr]{}{\langle \hat{c}_{\ul{X}} \circ \pi_1, \hat{c}_{\ul{X}} \circ \pi_2\rangle} & & & & \Ind(T)^2{\ul{X}} \times_{\Ind(T){\ul{X}}}\Ind(T)^2{\ul{X}} \ar[d]{}{(\hat{+} \ast \Ind(T))_{\ul{X}}} \\
	\Ind(T)^2{\ul{X}} \ar[rrrr, swap]{}{\hat{c}_{\ul{X}}} & & & & \Ind(T)^2{\ul{X}}
	\end{tikzcd}
	\]
	\[
	\begin{tikzcd}
	\Ind(T){\ul{X}} \ar[d, swap]{}{(\Ind(T) \ast 0)_{\ul{X}}} \ar[r, equals]{}{} & \Ind(T){\ul{X}} \ar[d]{}{(0 \ast \Ind(T))_{\ul{X}}} \\
	\Ind(T)^2{\ul{X}} \ar[r, swap]{}{\hat{c}_{\ul{X}}} & \Ind(T)^2{\ul{X}}
	\end{tikzcd}
	\]
	commute. The first and third diagrams are established similarly, so we need only establish the commutativity of the first and second diagrams to prove the proposition. Let $\ul{X} = (X_i)_{i \in I} = (X_i)$ be an object of $\Ind(\Cscr)$. For this we begin by establishing the commutativity of the first digaram. Note that on one hand
	\[
	\big(\Ind(T) \ast \Ind(p)\big)_{\ul{X}} = \big((T \ast p)_{X_i}\big) = \big((p \ast T)_{X_i} \circ c_{X_i}\big)
	\]
	because $\Cscr$ is a tangent category. On the other hand we calculate that
	\begin{align*}
	\big(\Ind(p) \ast \Ind(T)\big)_{\ul{X}} \circ \hat{c}_{\ul{X}} &= \big((p \ast T)_{X_i}\big) \circ \theta^{-1}_{\ul{X}} \circ \ul{c}_{\ul{X}} \circ \theta_{\ul{X}} = \big((p \ast T)_{X_i}\big) \circ \theta^{-1}_{\ul{X}} \circ (c_{X_i}) \circ \theta_{\ul{X}} \\
	&= \big((p \ast T)_{X_i}\big) \circ (c_{X_i}) \circ \theta^{-1}_{\ul{X}} \circ \theta_{\ul{X}} = \big((p \ast T)_{X_i}\big) \circ (c_{X_i}) \\
	&= \big((p \ast T)_{X_i} \circ c_{X_i}\big),
	\end{align*}
	which shows that the first diagram indeed commutes.
	
	To show the commutativity of the second diagram we note that on one hand a routine check shows
	\begin{align*}
	\hat{c}_{\ul{X}} \circ \big(\Ind(T) \ast \hat{+}\big)_{\ul{X}} &= \theta_{\ul{X}}^{-1} \circ (c_{X_i}) \circ \big((T \ast +)_{X_i}\big) \circ \theta_{\ul{X}}.
	\end{align*}
	On the other hand
	\begin{align*}
	&\big(\hat{+} \ast T\big)_{\ul{X}} \circ \left\langle  \hat{c}_{\ul{X}} \circ \pi_1, \hat{c}_{\ul{X}} \circ \pi_2 \right\rangle  \\
	&= \big(\hat{+} \ast T\big)_{\ul{X}} \circ  \left\langle \theta^{-1}_{\ul{X}} \circ (c_{X_i}) \circ \theta_{\ul{X}} \circ \pi_1, \theta^{-1}_{\ul{X}} \circ (c_{X_i}) \circ \theta_{\ul{X}} \circ \pi_2 \right\rangle \\
	&= \theta_{\ul{X}}^{-1} \circ \big((+ \ast T)_{X_i}\big) \circ \left\langle (c_{X_i}) \circ \theta_{\ul{X}} \circ \pi_1, (c_{X_i}) \circ \theta_{\ul{X}} \circ \pi_2 \right\rangle \\
	&= \theta_{\ul{X}}^{-1} \circ \big((+ \ast T)_{X_i}\big) \circ \left(\big\langle (c_{X_i}) \circ \pi_{1,i}, (c_{X_i}) \circ \pi_{2,i} \big\rangle\right) \circ \theta_{\ul{X}} \\
	&= \theta_{\ul{X}}^{-1} \circ \left((p \ast T)_{X_i} \circ \left\langle (c_{X_i}) \circ \pi_{1,i}, (c_{X_i}) \circ \pi_{2,i} \right\rangle\right) \circ \theta_{\ul{X}} \\
	&= \theta^{-1}_{\ul{X}} \circ \big(c_{X_i} \circ (T \ast +)_{X_i}\big) \circ \theta_{\ul{X}} = \theta_{\ul{X}}^{-1} \circ (c_{X_i}) \circ \big((T \ast +)_{X_i}\big) \circ \theta_{\ul{X}},
	\end{align*}
	which shows that the secon diagram commutes.
\end{proof}

From here we show the coherences that $\hat{c}$ and $\hat{\ell}$ satisfy.

\begin{lemma}\label{Lemma: Ind-cherences between lift and flip}
	Let $(\Cscr, \Tbb)$ be a tangent category. The ind-canonical flip $\hat{c}:\Ind(T)^2 \Rightarrow \Ind(T)^2$ is an involution, $\hat{c} \circ \hat{\ell} = \hat{\ell}$, and the diagrams
	\[
	\begin{tikzcd}
	\Ind(T) \ar[rr]{}{\hat{\ell}} \ar[d, swap]{}{\hat{\ell}} & & \Ind(T)^2 \ar[d]{}{\Ind(T) \ast \hat{\ell}}  \\
	\Ind(T)^2 \ar[rr, swap]{}{\hat{\ell} \ast \Ind(T)} & & \Ind(T)^3 
	\end{tikzcd}
	\]
	\[
	\begin{tikzcd}
	\Ind(T)^3 \ar[d, swap]{}{\hat{c} \ast \Ind(T)} \ar[rr]{}{\Ind(T) \ast \hat{c}} & & \Ind(T)^3 \ar[rr]{}{\hat{c} \ast \Ind(T)} & & \Ind(T)^3\ar[d]{}{\hat{c} \ast \Ind(T)} \\
	\Ind(T)^3 \ar[rr, swap]{}{\Ind(T) \ast \hat{c}} & & \Ind(T)^3 \ar[rr, swap]{}{\hat{c} \ast \Ind(T)} & & \Ind(T)^3
	\end{tikzcd}
	\]
	\[
	\begin{tikzcd}
	\Ind(T)^2 \ar[d, swap]{}{\hat{c}} \ar[rr]{}{\hat{\ell} \ast \Ind(T)} & & \Ind(T)^3 \ar[rr]{}{\Ind(T) \ast \hat{c}} & & \Ind(T)^3 \ar[d]{}{\hat{c} \ast \Ind(T)} \\
	\Ind(T)^2 \ar[rrrr, swap]{}{\Ind(T) \ast \hat{\ell}} & & & & \Ind(T)^3
	\end{tikzcd}
	\]
	commute.
\end{lemma}
\begin{proof}
	That $\hat{c}$ is an involution follows from the calculation, for any $\ul{X} = (X_i) \in \Ind(\Cscr)_0$,
	\[
	\hat{c}^2 = \theta^{-1} \circ (c_{X_i}) \circ \theta \circ \theta^{-1} \circ (c_{X_i}) \circ \theta`= \theta^{-1} \circ (c_{X_i})^2 \circ \theta = \theta^{-1} \circ \id_{\Ind(T^2)\ul{X}} \circ \theta = \id_{\Ind(T)^2\ul{X}}.
	\]
	The second identity follows from the calculation
	\begin{align*}
	\hat{c}_{\ul{X}} \circ \hat{\ell}_{\ul{X}} &= \theta^{-1}_{\ul{X}} \circ (c_{X_i}) \circ \theta_{\ul{X}} \circ \theta^{-1}_{\ul{X}} \circ  (\ell_{X_i}) = \theta_{\ul{X}}^{-1} \circ (c_{X_i}) \circ (\ell_{X_i}) = \theta_{\ul{X}}^{-1} \circ (c_{X_i} \circ \ell_{X_i}) \\
	&= \theta_{\ul{X}}^{-1} \circ (\ell_{X_i}) = \hat{\ell}_{\ul{X}}.
	\end{align*}
	Finally the verification of the commuting diagrams is tedious but straightforward check using the naturality of the isomorphisms $\Ind(T^n) \cong \Ind(T)^n$ together with the identities satisfied by the canonical flip and vertical lift in $\Cscr$.
\end{proof}

As the last necessary ingredient in showing that $(\Ind(\Cscr), \Ind(\Tbb))$, we will prove the universality of the ind-vertical lift. This proof will rely on ind-presheaves $\Ind_{\mathbf{PSh}}(\Cscr)$ and the functor $L:\Ind(\Cscr) \to [\Cscr^{\op},\Set]$ as it is convenient when proving properties about limits and colimits that need not hold on the nose, i.e., those properties which need only hold up to isomorphism.
\begin{proposition}\label{Prop: Universality of ind-vertical lift}
	Let $(\Cscr,\Tbb)$ be a tangent category and let $\ul{X} \in \Ind(\Cscr)_0$. Then the diagram
	\[
	\begin{tikzcd}
	\Ind(T)_2\ul{X} \ar[rrrrrrrr]{}{\big(\Ind(T) \ast \hat{+}\big)_{\ul{X}} \circ \left\langle \hat{\ell} \circ \pi_1, \big(\Ind(0) \ast T\big)_{\ul{X}} \circ \pi_2 \right\rangle} & & & & & & & &\Ind(T)^2\ul{X} \ar[rrrr, shift left = 0.5 ex]{}{\big(\Ind(T) \ast \Ind(p)\big)_{\ul{X}}} \ar[rrrr, shift right = 0.5 ex, swap]{}{\Ind(0)_{\ul{X}} \circ \Ind(p)_{\ul{X}} \circ \big(\Ind(p) \ast \Ind(T)\big)_{\ul{X}}} & & & & \Ind(T)\ul{X}
	\end{tikzcd}
	\]
	is an equalizer in $\Ind(\Cscr)$.
\end{proposition}
\begin{proof}
	Because being an equalizer is true along isomorphic objects in $\Cscr$, stable under equivalence of categories, and in light of Proposition \ref{Prop: L functor on ind is fully faithful}, it suffices to prove that the diagram
	\[
	\begin{tikzcd}
	L\Ind(T)_2\ul{X} \ar[rrrrrrrr]{}{L\big(\Ind(T) \ast \hat{+}\big)_{\ul{X}} \circ L\left\langle \hat{\ell} \circ \pi_1, \big(\Ind(0) \ast T\big)_{\ul{X}} \circ \pi_2 \right\rangle} & & & & & & & &L\Ind(T)^2\ul{X} \ar[rrrr, shift left = 0.5 ex]{}{L\big(\Ind(T) \ast \Ind(p)\big)_{\ul{X}}} \ar[rrrr, shift right = 0.5 ex, swap]{}{L\Ind(0)_{\ul{X}} \circ L\Ind(p)_{\ul{X}} \circ L\big(\Ind(p) \ast \Ind(T)\big)_{\ul{X}}} & & & & L\Ind(T)\ul{X}
	\end{tikzcd}
	\]
	is an equalizer in $[\Cscr^{\op},\Set]$. For this we calculate that the diagram above is is isomorphic to the diagram
	\[
	\begin{tikzcd}
	L\Ind(T)\ul{X} \times_{L\ul{X}} L\Ind(T)\ul{X} \ar[r] & L\Ind(T)^2\ul{X} \ar[r, shift left = 0.5 ex] \ar[r, shift left = -0.5 ex] & L\Ind(T)\ul{X}
	\end{tikzcd}
	\] 
	where the horizontal morphisms are the untangling of the $L$ morphisms defined above. By the fact that filtered colimits commute with finite limits in the presheaf topos $[\Cscr^{\op},\Set]$, we find that the diagram above is isomorphic to the diagram of presheaves:
	\[
	\begin{tikzcd}
	\lim\limits_{\substack{\longrightarrow \\ i \in I}} \Cscr(-,T_2X_i) \ar[r] & \lim\limits_{\substack{\longrightarrow \\ i \in I}} \Cscr(-,T^2X_i) \ar[r, shift left = 0.5 ex]{}{} \ar[r, shift left = -0.5ex, swap]{}{} & \lim\limits_{\substack{\longrightarrow \\ i \in I}} \Cscr(-,TX_i)
	\end{tikzcd}
	\]
	Finally, each $i$-indexed component of the diagram above,
	\[
	\begin{tikzcd}
	\Cscr(-,T_2X_i) \ar[r] & \Cscr(-,T^2X_i) \ar[r, shift left = 0.5 ex]{}{} \ar[r, shift left = -0.5ex, swap]{}{} & \Cscr(-,TX_i)
	\end{tikzcd}
	\]
	is an equalizer by the fact that the Yoneda Lemma is continuous and the fact that $(\Cscr, \Tbb)$ is a tangent category. This implies in turn that the diagram
	\[
	\begin{tikzcd}
	\lim\limits_{\substack{\longrightarrow \\ i \in I}} \Cscr(-,T_2X_i) \ar[r] & \lim\limits_{\substack{\longrightarrow \\ i \in I}} \Cscr(-,T^2X_i) \ar[r, shift left = 0.5 ex]{}{} \ar[r, shift left = -0.5ex, swap]{}{} & \lim\limits_{\substack{\longrightarrow \\ i \in I}} \Cscr(-,TX_i)
	\end{tikzcd}
	\]
	and hence
	\[
	\begin{tikzcd}
	L\Ind(T)_2\ul{X} \ar[rrrrrrrr]{}{L\big(\Ind(T) \ast \hat{+}\big)_{\ul{X}} \circ L\left\langle \hat{\ell} \circ \pi_1, \big(\Ind(0) \ast T\big)_{\ul{X}} \circ \pi_2 \right\rangle} & & & & & & & &L\Ind(T)^2\ul{X} \ar[rrrr, shift left = 0.5 ex]{}{L\big(\Ind(T) \ast \Ind(p)\big)_{\ul{X}}} \ar[rrrr, shift right = 0.5 ex, swap]{}{L\Ind(0)_{\ul{X}} \circ L\Ind(p)_{\ul{X}} \circ L\big(\Ind(p) \ast \Ind(T)\big)_{\ul{X}}} & & & & L\Ind(T)\ul{X}
	\end{tikzcd}
	\]
	are both equalizers in $[\Cscr^{\op},\Set]$. Finally, appealing Proposition to \ref{Prop: L functor on ind is fully faithful} proves the result.
\end{proof}

We now have the ingredients to show that $\Ind(\Cscr)$ is a tangent category whenever $\Cscr$ is a tangent category. 
\begin{Theorem}\label{Thm: Ind Tangent Category}
	Let $(\Cscr,\Tbb)$ be a tangent category. Then the category $(\Ind(\Cscr),\Ind(\Tbb))$ is a tangent category where $\Ind(\Tbb)$ is the tangent structure
	\[
	\Ind(\Tbb) := (\Ind(T), \Ind(p), \Ind(0), \hat{+}, \hat{\ell}, \hat{c})
	\]
	where $\Ind(T)$ is the indicization of $T$ induced by Proposition \ref{Prop: Ind functors} and $\Ind(p), \Ind(0), \hat{+}, \hat{\ell},$ and $\hat{c}$ are the natural transformations constructed in Lemmas \ref{Lemma: Existence of Indp and Ind0}, \ref{Lemma: Existence of ind-bundle addition}, \ref{Lemma: Existence of Ind-vertical lift}, and Lemma \ref{Lemma: Existence of ind-canonical flip}, respectively.
\end{Theorem}
\begin{proof}
	We now estalish that $(\Ind(\Cscr),\Ind(\Tbb))$ satisfies the axioms of Definition \ref{Defn: Tangent Categpry}:
	\begin{enumerate}
		\item That $\Ind(\Cscr)$ admits all tangent pullbacks folloes from \cite[Proposition I.8.9.1.c]{SGA4} and that $\Ind(T)$ and its compositional powers preserve the pullback powers $\Ind(T)_m$ is Corollary \ref{Cor: Ind Tangent functor commutes with Ind tangent pullbacks}.
		\item  The ind-zero transformation $\Ind(0)$ is constructed in Lemma \ref{Lemma: Existence of Indp and Ind0} and the ind-addition transformation is constructed in Lemma \ref{Lemma: Existence of ind-bundle addition}. That each object $\Ind(p)_{\ul{X}}:\Ind(T)\ul{X} \to \ul{X}$ is a commutative monoid in $\Ind(\Cscr)_{/\ul{X}}$ is Proposition \ref{Prop: Additive Ind Bundle}.
		\item The ind-vertical lift $\hat{\ell}$ exists by Lemma \ref{Lemma: Existence of Ind-vertical lift} and the pair $(\hat{\ell}_{\ul{X}},\Ind(0)_{\ul{X}})$ is a bundle morphism for any $\ul{X}$ in $\Ind(\Cscr)_0$ by Proposition \ref{Prop: Ind vertical lift is part of a bundle map}.
		\item The existence of the ind-canonical flip follows from Lemma \ref{Lemma: Existence of ind-canonical flip} while the fact that $(\hat{c}_{\ul{X}},\id_{\Ind(T)\ul{X}})$ is a bundle map for any $\ul{X}$ in $\Ind(\Cscr)_0$ follows from Proposition \ref{Prop: Ind-canonical flip is bundle map}.
		\item The coherences between the ind-canonical flip and ind-vertical lift are given in Lemma \ref{Lemma: Ind-cherences between lift and flip}.
		\item The universality of the ind-vertical lift is proved in Proposition \ref{Prop: Universality of ind-vertical lift}.
	\end{enumerate}
	Because each axiom of Definition \ref{Defn: Tangent Categpry} is satisfied it follows that $(\Ind(\Cscr), \Ind(\Tbb))$ is a tangent category.
\end{proof}

We now show the functoriality of the $\Ind$-construction on tangent morphisms. 

\begin{Theorem}\label{Thm: Functoriality of ind-category}
Let $(F,\alpha):(\Cscr, \Tbb) \to (\Dscr, \Sbb)$ be a morphism of tangent categories. Then the induced map $(\Ind(F), \hat{\alpha}):(\Ind(\Cscr), \Ind(\Tbb)) \to (\Ind(\Dscr), \Ind(\Sbb))$ is a morphism of tangent categories where $\hat{\alpha}$ is the natural transformation:
\[
\begin{tikzcd}
\Ind(\Cscr) \ar[rr, bend left = 95, ""{name = UU}]{}{\Ind(F) \circ \Ind(T)} \ar[rr, bend left = 25, ""{name = U}]{}[description]{\Ind(F \circ T)} \ar[rr, bend right = 15, ""{name = L}]{}[description]{\Ind(S \circ F)} \ar[rr, bend right = 70, swap, ""{name = LL}]{}{\Ind(S) \circ \Ind(F)} & & \Ind(\Dscr) \ar[from = UU, to = U, Rightarrow, shorten >= 4pt, shorten <= 4pt]{}{\phi_{T, F}} \ar[from = U, to = L, Rightarrow, shorten >= 4pt, shorten <= 4pt]{}{\Ind(\alpha)} \ar[from = L, to = LL, Rightarrow, shorten >= 4pt, shorten <= 4pt]{}{\phi_{F, S}^{-1}}
\end{tikzcd}
\]
Furthermore, $(\Ind(F), \hat{\alpha})$ is a strong tangent morphism if and only if $(F, \alpha)$ is a strong tangent morphism.
\end{Theorem}
\begin{proof}
The fact that $(\Ind(F), \hat{\alpha})$ is a tangent morphism is a straightforward but tedious $2$-categorical verification; we illustrate the first such verification and omit the rest. To establish the first diagram of functors and natural transformations we simply paste the commuting triangles as in the diagram below
\[
\begin{tikzcd}
\Ind(F) \circ \Ind(T) \ar[r]{}{\phi_{F,T}} \ar[drr, swap, bend right = 20]{}{\Ind(F) \ast \Ind(p)} & \Ind(F \circ T) \ar[dr]{}[description]{\Ind(F \ast p)} \ar[rr]{}{\Ind(\alpha)} & & \Ind(S \circ F) \ar[r]{}{\phi_{F,S}^{-1}} \ar[dl]{}[description]{\Ind(q \ast F)} & \Ind(S) \circ \Ind(F) \ar[dll, bend left = 20]{}{\Ind(q) \ast \Ind(p)} \\
 & & \Ind(F) & 
\end{tikzcd}
\]
to establish that
\[
\begin{tikzcd}
\Ind(F) \circ \Ind(T) \ar[r]{}{\hat{\alpha}} \ar[dr, swap]{}{\Ind(F) \ast \Ind(p)} & \Ind(S) \circ \Ind(F) \ar[d]{}{\Ind(q) \ast \Ind(F)} \\
 & \Ind(F)
\end{tikzcd}
\]
commutes. The remaining four axioms are verified and established similarly and so are omitted.

For the final claim regarding detecting when $(F,\alpha)$ is strong: 

$\implies:$ assume that $(\Ind(F), \hat{\alpha})$ is a strong tangent morphism. Then it is strong on all constant objects $\ul{X} = (X)$ for $X \in \Cscr_0$; as such it follows that $\alpha_{X}$ is an isomorphism for every $X \in \Cscr_0$. That $F$ preserves tangent pullbacks and equalizers is shown similarly. 

$\impliedby:$ Assume  $\alpha$ is a natural isomorphism so that $\Ind(\alpha)$ and $\hat{\alpha} = \phi_{F,S}^{-1} \circ \Ind(\alpha) \circ \phi_{F,S}$ are isomorphisms as well. Similarly, if $F$ preserves tangent pullbacks and equalizers, so does $\Ind(F)$.
\end{proof}
\begin{corollary}\label{Cor: Commute strictly diagram}
The diagram
\[
\begin{tikzcd}
\fTan \ar[r]{}{\Ind} \ar[d, swap]{}{\Forget} & \fTan \ar[d]{}{\Forget} \\
\fCat \ar[r, swap]{}{\Ind} & \fCat
\end{tikzcd}
\]
commutes strictly.
\end{corollary}
\begin{proof}
This is immediate after unwinding the definitions.
\end{proof}

We now close this section by proving that while the $\Ind$ pseudofunctor does not commute (even up to isomorphism) with the free tangent functor, there is a pseudonatural transformation:
\[
\begin{tikzcd}
\fCat \ar[r, ""{name = U}]{}{\Ind} \ar[d, swap]{}{\Free} & \fCat \ar[d]{}{\Free} \\
\fTan \ar[r, swap, ""{name = D}]{}{\Ind} & \fTan \ar[from = U, to = D, Rightarrow, shorten <= 4pt, shorten >= 4pt]{}{\alpha}
\end{tikzcd}
\]
To do this, however, we necessitate a discussion of what it means to be a free tangent functor, i.e., we need to at least give a short description of the $2$-functor $\Free:\fCat \to \fTan$.

In \cite{LeungWeil} P.\@ Leung showed that the category $\Weil_1$ of Weil algebras constitute the free tangent structure in a very precise sense (we will recall and see some of this later, but one of the main results in \cite{LeungWeil} is that tangent structures on a category $\Cscr$ are equivalent to Weil actegory structures on $\Cscr$). While we will get to know this category in more detail later (by comparing and contrasting the exposition of \cite{LeungWeil} and what is presented in B.\@ MacAdam's thesis \cite{BenThesis}), for now we just need to know the following result: for any category $\Cscr$, the free tangent category over $\Cscr$ is $\Free(\Cscr) := \Weil_1 \times \Cscr$. We will briefly recall the definition of $\Weil_1$, however.

\begin{definition}[{\cite[Definition 4.2]{LeungWeil}}]
The category $\Weil_1$ is defined as follows:
\begin{itemize}
	\item Objects: For each $n \in \N$, define the following crigs:
	\[
	W^n := \begin{cases}
	\N & \text{if}\, n = 0; \\
	\N[x_1, \cdots, x_n]/(x_ix_j:1 \leq i, j \leq n) & \text{if}\, n \geq 1.
	\end{cases}
	\]
	Then the objects of $\Weil_1$ are the closure of $\lbrace W^n \; | \; n \in \N \rbrace$ under finite coproducts $\otimes_{\N}$ of crigs.
	\item Morphisms: Morphisms $\varphi:A \to B$ of crigs for which the canonical maps
	\[
	\begin{tikzcd}
	A \ar[dr, swap]{}{\pi_{\mfrak_{A}}} \ar[rr]{}{\varphi} & & B \ar[dl]{}{\pi_{\mfrak_{B}}} \\
	 & \N
	\end{tikzcd}
	\]
	commute; note that in each case $\mfrak_{A}$ and $\mfrak_{B}$ are the unique maximal ideals of $A$ and $B$, respectively.
\end{itemize}
\end{definition}
\begin{Theorem}[{\cite{LeungWeil}; \cite[Observation 4.2.5]{BenThesis}}]
For any category $\Cscr$, there is a free tangent category given by $\Weil_1 \times \Cscr$.
\end{Theorem}

Because the product construction is $2$-functorial we find immediately that there is a functor $\Free:\fCat \to \fTan$ which sends a category $\Cscr$ to the free tangent category over itself: $\Free(\Cscr) := \Weil_1 \times \Cscr$. It acts on $1$-morphisms by $\Free(f) := \id_{\Weil_1} \times f$ and on $2$-morphisms by $\Free(\alpha) := \iota_{\id_{\Weil_1}}\times \alpha$. To construct our pseudonatural transformation
\[
\begin{tikzcd}
\fCat \ar[r, ""{name = U}]{}{\Ind} \ar[d, swap]{}{\Free} & \fCat \ar[d]{}{\Free} \\
\fTan \ar[r, swap, ""{name = D}]{}{\Ind} & \fTan \ar[from = U, to = D, Rightarrow, shorten <= 4pt, shorten >= 4pt]{}{\alpha}
\end{tikzcd}
\]
we compute that on one hand
\[
\Ind(\Free(\Cscr)) = \Ind(\Weil_1 \times \Cscr)
\]
while on the other hand
\[
\Free(\Ind(\Cscr)) = \Weil_1 \times \Ind(\Cscr).
\]
To construct our pseudonatural transformation $\alpha$ we first note that by \cite[Proposition 6.1.12]{KashiwaraSchapira} for any category $\Cscr$ there is a natural equivalence of categories
\[
\Ind(\Free(\Cscr)) = \Ind(\Weil_1 \times \Cscr) \simeq \Ind(\Weil_1) \times \Ind(\Cscr).
\]
Now consider the embedding $\operatorname{incl}_{\Weil_1}:\Weil_1 \to \Ind(\Weil_1)$ which sends an object $X$ to the functor\footnote{Recall that $\mathbbm{1}$ is the terminal category, i.e., the category with one object (which we label as $\ast$) and one morphism (namely $\id_{\ast}$).} $\ul{X}:\mathbbm{1} \to \Cscr$ given by $\ast \mapsto X, \id_{\ast} \mapsto \id_X$ and sends a morphism $f:X \to Y$ to the corresponding natural transformation $\ul{f}:\ul{X} \to \ul{Y}$. We then define $\alpha_{\Cscr}$ to be the composite:
\[
\begin{tikzcd}
\Weil_1 \times \Ind(\Cscr) \ar[rrr]{}{\operatorname{incl}_{\Weil_1} \times \id_{\Ind(\Cscr)}} & & & \Ind(\Weil_1) \times \Ind(\Cscr) \ar[r]{}{\simeq} & \Ind(\Weil_1 \times \Cscr)
\end{tikzcd}
\]
To see how to define the witness transformations let $F:\Cscr \to \Dscr$ be a functor. We then get the pasting diagram
\[
\begin{tikzcd}
\Weil_1 \times \Ind(\Cscr) \ar[dd, bend right = 100, swap]{}{\alpha_{\Cscr}} \ar[rrrr, ""{name = Up}]{}{\id_{\Weil_1} \times \Ind(F)} \ar[d, swap]{}{\operatorname{incl}_{\Weil_1} \times \id} & & & & \Weil_1 \times \Ind(\Dscr) \ar[d]{}{\operatorname{incl}_{\Weil_1} \times \id} \ar[dd, bend left = 100]{}{\alpha_{\Dscr}} \\
\Ind(\Weil_1) \times \Ind(\Cscr) \ar[rrrr, ""{name = Mid}]{}[description]{\id_{\Ind(\Weil)_1} \times \Ind(F)} \ar[d, swap]{}{\simeq} & & & & \Ind(\Weil_1) \times \Ind(\Dscr) \ar[d]{}{\simeq} \\
\Ind(\Weil_1 \times \Cscr) \ar[rrrr, swap, ""{name = Down}]{}{\Ind(\Weil_1 \times F)} & & & & \Ind(\Weil_1 \times \Dscr) \ar[from = Up, to = Mid, equals, shorten <= 4pt, shorten >= 4pt] \ar[from = Mid, to = Down, Rightarrow, shorten <= 8pt, shorten >= 4pt]{}{\cong}
\end{tikzcd}
\]
where the natural isomorphism in the bottom-most square is induced from the natural equivalence of \cite[Proposition 6.2.11]{KashiwaraSchapira}. Note this pastes to the invertible $2$-cell
\[
\begin{tikzcd}
(\Free \circ \Ind)(\Cscr) \ar[rr, ""{name = Up}]{}{(\Free \circ \Ind)(F)} \ar[d, swap]{}{\alpha_{\Cscr}} & & (\Free \circ \Ind)(\Dscr) \ar[d]{}{\alpha_{\Dscr}} \\
(\Ind \circ \Free)(\Cscr) \ar[rr, swap, ""{name = Below}]{}{(\Ind \circ \Free)(F)} & & (\Ind \circ \Free)(\Dscr) \ar[from = Up, to = Below, Rightarrow, shorten <= 4pt, shorten >= 4pt]{}{\cong}
\end{tikzcd}
\]
which we define to be our witness transformation $\alpha_f$. From here the pseudonaturality of $\alpha$ is routine but straightforward to check: the compatibility with the compositors for $\Free(-)$ and $\Ind(-)$ more or less comes down to the fact that the equivalence $\Ind(\Weil_1 \times \Cscr) \simeq \Ind(\Weil_1) \times \Ind(\Cscr)$ is natural and the action of $\Ind$ on $\fCat$ occurs only in the right-handed variable and does not interact with the $\Weil_1$-variable otherwise. This leads to the following proposition.
\begin{proposition}
There is a pseudonatural transformation
\[
\begin{tikzcd}
\fCat \ar[r, ""{name = U}]{}{\Ind} \ar[d, swap]{}{\Free} & \fCat \ar[d]{}{\Free} \\
\fTan \ar[r, swap, ""{name = D}]{}{\Ind} & \fTan \ar[from = U, to = D, Rightarrow, shorten <= 4pt, shorten >= 4pt]{}{\alpha}
\end{tikzcd}
\]
where the transition functors are given by
\[
\begin{tikzcd}
\Weil_1 \times \Ind(\Cscr) \ar[rrr]{}{\operatorname{incl}_{\Weil_1} \times \id_{\Ind(\Cscr)}} & & & \Ind(\Weil_1) \times \Ind(\Cscr) \ar[r]{}{\simeq} & \Ind(\Weil_1 \times \Cscr)
\end{tikzcd}
\]
\end{proposition}
\begin{remark}
Note that since $\Weil_1 \not\simeq \Ind(\Weil_1)$\footnote{Since $\Ind(\Weil_1)$ admits all small filtered colimits, we only need to demonstrate that there is a filtered colimit which $\Weil_1$ does not admit. However, $\Weil_1$ does not have infinite colimits as the countably infinite coproduct $X = \bigotimes_{n \in \N} \N[x_n]/(x_n^2) \cong \N[x_n : n \in \N]/(x_n^2 : n \in \N)$ is not an object in $\Weil_1$. As such we get that $\Weil_1 \not\simeq \Ind(\Weil_1)$.} it is impossible for $\alpha$ to be taken to be an invertible $2$-cell as $\alpha_{\Cscr}$ is an equivalence if and only if $\Cscr = \emptyset$ is the empty category.
\end{remark}

\section{Properties of Ind-Tangent Categories}
In this section we examine how the $\Ind$-construction interacts with various important and standard constructions in the tangent category world and in algebraic geometry. We in particular give a pseudolimit characterization of the slice category $\Ind(\Cscr)_{X}$ where $X = (X_i)$ is an $\Ind$-object, we exampine the $\Ind$-tangent category of a representable tangent category, how $\Ind$ interacts with differential objects (and more generally how $\Ind$ interacts with differential bundles), and finally how $\Ind$ interacts with Cartesian differential categories. To do these investigations we need some basic structural results regarding the $\Ind$-pseudofunctor and its preservation of adjoints. Before proceeding, however, let us recall the definition of a representable tangent category.
\subsection{A Pseudolimit Result}
In this short subsection we show that there is a sensible characterization of the slice category $\Ind(\Cscr)_{/(X_i)}$ (at least when $\Cscr$ has finite pullbacks) in terms of a pseudolimit over the diagram of categories $\Ind(\Cscr)_{/X_i}$ with morphism functors induced by the pullback functors. While we do not explicitly use this in this article, I anticipate that this will have future uses when studying the tangent category of formal schemes explicitly and also in certain functional-analytic characterizations (cf.\@ Subsection \ref{Subsection: Smooth} below) involving slice tangent categories.

\begin{proposition}\label{Prop: A pseudolimit thing}
	Let $\Cscr$ be a finitely complete category and let $\Xfrak = (X_i)$ be an object in $\Ind(\Cscr)$. Consider the pseudofunctor $G:I^{\op} \to \fCat$ given by:
	\begin{itemize}
		\item The object assignment sends $i \in I_0$ to the category $\Ind(\Cscr)_{/X_i}$.
		\item The morphism assignment sends $\varphi:i \to i^{\prime}$ to the pullback functor $\tilde{\varphi}^{\ast}:\Ind(\Cscr)_{/X_{i^{\prime}}} \to \Ind(\Cscr)_{X_i}$, i.e., the functor which sends an object $Z \to X_{i^{\prime}}$ to a chosen pullback
		\[
		\begin{tikzcd}
			Z \times_{X_{i^{\prime}}} X_i \ar[r] \ar[d] & Z \ar[d] \\
			X_{i} \ar[r, swap]{}{\tilde{\varphi}} & X_{i^{\prime}}
		\end{tikzcd}
		\]
		where $\tilde{\varphi}:X_i \to X_{i^{\prime}}$ is the structure map induced from $\Xfrak$.
		\item The compositors are induced by the natural isomorpshims of pullbacks.
	\end{itemize}
	Then $\Ind(\Cscr)_{/\Xfrak}$ is the pseudolimit of the diagram in $\fCat$ induced by $F$.
\end{proposition}
\begin{proof}
	Begin by noting that $\Xfrak$ is the filtered colimit of the $X_i$ in $\Ind(\Cscr)$; this follows from the fact that under the equivalence $\Ind(\Cscr) \simeq \Ind_{\mathbf{PSh}}(\Cscr)$, $\Xfrak$ is identified with the presheaf
	\[
	\colim_{i \in I}\yon(X_i)
	\]
	in $\Ind_{\mathbf{PSh}}(\Cscr)$ and equivalences preserve colimits. Consequently we write $\alpha_i:X_i \to \Xfrak$ for the colimit structure maps for all $i \in I$.
	
	We now will establish that any diagram factors through $\Ind(\Cscr)_{/\Xfrak}$. Let $\Dscr$ be a category such that there are functors $F_i:\Dscr \to \Cscr_{/X_i}$ for all $i \in I_0$ such that for any $\varphi:i \to i^{\prime}$ in $I$, there is an invertible $2$-cell
	\[
	\begin{tikzcd} 
		& \Ind(\Cscr)_{/X_{i^{\prime}}} \ar[dd]{}{\tilde{\varphi}^{\ast}} \\ 
		\Dscr \ar[ur, ""{name = U}]{}{F_{i^{\prime}}} \ar[dr, swap, ""{name = D}]{}{F_i} & \\
		& \Ind(\Cscr)_{/X_i} \ar[from = U, to = D, Rightarrow, shorten <= 4pt, shorten >= 4pt]{}{\rho_{\varphi}}\ar[from = U, to = D, Rightarrow, swap, shorten <= 4pt, shorten >= 4pt]{}{\cong}
	\end{tikzcd}
	\]
	Now observe that the natural isomorphism $\rho$ implies that in $\Cscr$ there is an isomorphism of functors
	\[
	F_i(-) \cong (\tilde{\varphi}^{\ast} \circ F_{i^{\prime}})(-) = F_{i^{\prime}}(-) \times_{X_{i^{\prime}}} X_i
	\]
	where the pullback is taken as in the diagram:
	\[
	\begin{tikzcd}
		F_{i^{\prime}}(-) \times_{X_{i^{\prime}}} X_i \ar[r]{}{p_0} \ar[d, swap]{}{p_1} & F_{i^{\prime}}(-) \ar[d]{}{} \\
		X_i \ar[r, swap]{}{\tilde{\varphi}} & X_{i^{\prime}}
	\end{tikzcd}
	\]
	Consequently taking the first projection of the pullback above and pre-composing with the isomorphism $\rho$ gives a structure natural transformation
	\[
	\begin{tikzcd}
		F_{i} \ar[r]{}{\cong} \ar[dr, swap]{}{\hat{\varphi}} & F_{i^{\prime}} \times_{X_{i^{\prime}}} X_i \ar[d]{}{p_0} \\
		& F_{i^{\prime}}
	\end{tikzcd}
	\]
	Taking the colimit of the $\hat{\varphi}$ for every object $d \in \Dscr_0$, which exists in $\Ind(\Cscr)$ because $I$ is a filtered category, produces a family of objects which we define by
	\[
	F(d) := \colim_{i \in I} F_i(d).
	\]
	The universal property of the colimit allows us to extend this to a functor $\tilde{F}:\Dscr \to \Ind(\Cscr)$; we claim that this extension naturally factors through the forgetful functor $\Ind(\Cscr)_{/\Xfrak} \to \Ind(\Cscr)$. To see this, however, we simply use the universal property of the colimits in sight to produce the commuting
	\[
	\begin{tikzcd}
		& F(d) \\
		F_{i}(d) \ar[rr, near start]{}{\hat{\varphi}_d} \ar[ur]{}{\beta_i} \ar[d]  & & F_{i^{\prime}}(d) \ar[ul, swap]{}{\beta_{i^{\prime}}}  \ar[d] \\
		X_{i} \ar[dr, swap]{}{\alpha_i} \ar[rr, near start, swap]{}{\tilde{\varphi}} & & X_{i^{\prime}} \ar[dl]{}{\alpha_{i^{\prime}}} \\
		& \Xfrak \ar[from = 1-2, to = 4-2, crossing over, dashed]{}{\exists!\theta}
	\end{tikzcd}
	\] 
	diagram with $\theta$ defined by the universal property of $F(d)$. This gives us the functor $F:\Dscr \to \Ind(\Cscr)_{/\Xfrak}$.
	
	We now claim that this functor $\theta$ allows us to factor the $2$-cells $\rho_{\varphi}$ in the sense that there is a pasting diagram
	\[
	\begin{tikzcd}
		& & \Ind(\Cscr)_{/X_{i^{\prime}}} \ar[dd]{}{\tilde{\varphi}^{\ast}} \\
		\Dscr\ar[urr, bend left = 20, ""{name = URR}]{}{F_{i^{\prime}}} \ar[drr, bend right = 20, swap, ""{name = DRR}]{}{F_{i}} \ar[r]{}{\theta} & \Ind(\Cscr)_{/\Xfrak} \ar[ur, ""{name = UR}]{}{\alpha_{i^{\prime}}^{\ast}} \ar[dr, ""{name = DR}, swap]{}{\alpha_i^{\ast}} \\
		& & \Ind(\Cscr)_{/X_{i}} \ar[from = URR, to = 2-2, swap, Rightarrow, shorten <= 4pt, shorten >= 4pt]{}{\gamma} \ar[from = URR, to = 2-2, Rightarrow, shorten <= 4pt, shorten >= 4pt]{}{\cong} \ar[from = 2-2, to = DRR, Rightarrow, shorten >=4pt, shorten <= 4pt]{}{\delta}  \ar[from = 2-2, to = DRR, Rightarrow, swap, shorten >=4pt, shorten <= 4pt]{}{\cong} \ar[from = UR, to = DR, Rightarrow, shorten <= 24pt, shorten <= 8pt]{}{\phi_{\widetilde{\varphi}, \alpha_{i^{\prime}}}} \ar[from = UR, to = DR, Rightarrow, shorten <= 24pt, shorten <= 8pt, swap]{}{\cong}
	\end{tikzcd}
	\]
	which pastes to $\rho_{\varphi}$. To see this we first will show the existence of $\gamma$. To this end note that on one hand, for any $d \in \Dscr_0$ we have
	\[
	(\alpha_{i}^{\ast} \circ \theta)(d) = F(d) \times_{\Xfrak} X_i.
	\]
	Now note that $X_i$ as a constant object is isomorphic in $\Ind(\Cscr)$ to the $I$-indexed object $\ul{X_{i^{\prime}}}:I \to \Cscr$ where $\ul{X_{i^{\prime}}}(x) := X_{i^{\prime}}$ for all objects and $\ul{X_{i^{\prime}}}(\varphi) = \id_{X_{i^{\prime}}}$ for all morphisms. Under this 
	representation we get that $X_{i^{\prime}} \cong \colim_{I}\ul{X_{i^{\prime}}}(i)$ and so
	\[
	(\alpha_{i^{\prime}}^{\ast} \circ \theta)(d) = F(d) \times_{\Xfrak} X_i \cong F(d) \times_{\Xfrak} \left(\colim_{i \in I} \ul{X_{i^{\prime}}}(i)\right) \cong \left(\colim_{I} F_i(d)\right) \times_{\colim_{i \in I} X_i} \left(\colim_{i \in I}\ul{X_{i^{\prime}}}(i)\right).
	\]
	Using now that each filtered colimit in sight is taken over the same (filtered indexing category and that filtered colimits commute naturally with pullbacks in $\Ind(\Cscr)$ we find
	\begin{align*}
		(\alpha_{i^{\prime}}^{\ast} \circ \theta)(d) & \cong \left(\colim_{I} F_i(d)\right) \times_{\colim_{i \in I} X_i} \left(\colim_{i \in I}\ul{X_{i^{\prime}}}(i)\right) \cong \colim_{i \in I}\left(F_i(d) \times_{X_i} \ul{X_{i^{\prime}}}(i)\right) = \colim_{i \in I} \left(F_i(d) \times_{X_i} X_{i^{\prime}}\right)
	\end{align*}
	Note that since there is a natural isomorphism $F_i(-) \cong F_{i^{\prime}}(-) \times_{X_{i^{\prime}}} X_i$ whenever there is a morphism $\tilde{\varphi}:i \to i^{\prime}$ in $\Ind(\Cscr)$. Using that $I$ is filtered and that we can compute the colimit after applying the $d$-component of this natural isomorphism allows us to deduce that
	\[
	\colim_{i \in I} \left(F_i(d) \times_{X_i} X_{i^{\prime}}\right) \cong \colim_{i \in I}\left(\big(F_{i^{\prime}}(d) \times_{X_{i^{\prime}}} X_i\big) \times_{X_i} X_{i^{\prime}}\right) \cong \colim_{i \in I}\left(F_{i^{\prime}}(d)\right) \cong F_{i^{\prime}}(d)
	\]
	because the final colimit is taken in a variable which is not present in the object $F_{i^{\prime}}(d)$. Note that each isomorphism presented is natural; defining $\gamma$ to be the composite of these gives our desired $2$-cell:
	\[
	\begin{tikzcd}
		& \Ind(\Cscr)_{/\Xfrak} \ar[dr]{}{\alpha_{i^{\prime}}^{\ast}} \\
		\Dscr \ar[ur]{}{\theta} \ar[rr, swap]{}{F_{i^{\prime}}} & {} & \Ind(\Cscr)_{/X_{i^{\prime}}} \ar[from = 1-2, to = 2-2, Rightarrow, shorten <= 4pt, shorten >= 4pt]{}{\gamma} \ar[from = 1-2, to = 2-2, Rightarrow, shorten <= 4pt, shorten >= 4pt, swap]{}{\cong}
	\end{tikzcd}
	\]
	Establishing the existence of $\delta$ is similar to establishing that of $\gamma$ and omitted. Finally the existence of the $2$-cell follows immediately from the fact that $\Xfrak \cong \colim_{I} X_i$. That the resulting pasting diagram pastes to $\rho_{\varphi}$ is also an extremely tedious but straightforward argument that uses the fact that $\gamma$ and $\delta$ involve $\rho_{\varphi}$ in establishing the natural isomorphism and is also omitted.
	
	We now finally establish the remaining property of $\Ind(\Cscr)_{/\Xfrak}$ as a pseudolimit of the diagram of shape $G$. That is, assume that we have two functors $F, F^{\prime}:\Dscr \to \Ind(\Cscr)_{/\Xfrak}$ such that for all $i \in I_0$ there is a $2$-cell
	\[
	\begin{tikzcd}
		\Dscr \ar[rrr, ""{name = U}]{}{F} \ar[d, swap]{}{F^{\prime}} & & & \Ind(\Cscr)_{/\Xfrak} \ar[d]{}{\alpha_i^{\ast}} \\
		\Ind(\Cscr)_{/\Xfrak} \ar[rrr, swap, ""{name = D}]{}{\alpha_i^{\ast}} & & & \Ind(\Cscr)_{/X_i} \ar[from = U, to = D, Rightarrow, shorten <= 4pt, shorten >= 4pt]{}{\gamma_i}
	\end{tikzcd}
	\]
	for which the diagram of functors and natural transformations, for any $\varphi:i \to i^{\prime}$ in $I$,
	\[
	\begin{tikzcd}
		\tilde{\varphi}^{\ast} \circ \alpha_{i^{\prime}}^{\ast} \circ F \ar[r]{}{\tilde{\varphi}^{\ast} \ast \gamma_{i^{\prime}}} \ar[d, swap]{}{\phi_{\tilde{\varphi}, \alpha_{i^{\prime}}} \ast F} & \tilde{\varphi}^{\ast} \circ \alpha_{i^{\prime}}^{\ast} \circ F^{\prime} \ar[d]{}{\phi_{\tilde{\varphi}, \alpha_{i^{\prime}}} \ast F^{\prime}} \\
		\alpha_i^{\ast} \ar[r, swap]{}{\alpha_i^{\ast} \ast \gamma_i} & \alpha_i^{\ast} \circ F^{\prime}
	\end{tikzcd}
	\]
	commutes. We must establish that there is a unique $2$-cell $\gamma:F \to F^{\prime}$ making $\alpha_{i}^{\ast} \ast \gamma = \gamma_i$ for all $i \in I_0$. However, note that $\alpha_i^{\ast} \circ F = F(-) \times_{\Xfrak} X_i$ and $\alpha_i^{\ast} \circ F^{\prime} = F^{\prime}(-) \times_{\Xfrak} X_{i}$ for all $i \in I_0$ and the conditions on the naturality of $\gamma_i$ and the coherences regarding the interchange of the $\gamma_i$ with the pullbacks $\tilde{\varphi}^{\ast}$ tell us that we have a uniquely determined natural transformation
	\[
	\colim_{I}\left(F(-) \times_{\Xfrak} X_i\right) \xrightarrow{\colim_I(\gamma_i)} \colim_{I}\left(F^{\prime}(-) \times_{\Xfrak} X_i\right).
	\]
	Using the natural isomorphisms $\colim_{I}(F(-) \times_{\Xfrak} X_i) \cong F(-)$ and $\colim_{I}(F^{\prime}(-) \times_{\Xfrak} X_i)$ induced by the fact that filtered colimits commute with pullbacks in $\Ind(\Cscr)$ gives our desired natural transformation:
	\[
	\begin{tikzcd}
		\colim_{I}\left(F(-) \times_{\Xfrak} X_i\right) \ar[d, swap]{}{\colim_I(\gamma_i)} \ar[r]{}{\cong} & F(-) \ar[d]{}{\gamma} \\
		\colim_{I}\left(F^{\prime}(-) \times_{\Xfrak} X_i\right) \ar[r, swap]{}{\cong} & F^{\prime}(-)
	\end{tikzcd}
	\]
	Finally that this satisfies $\alpha_i^{\ast} \ast \gamma = \gamma_i$ is immediate by construction.
\end{proof}

\subsection{Representable Ind-Tangent Categories}
\begin{definition}[{\cite[Definition 3]{GarnerEmbedding}}]
Let $(\Cscr,\Tbb)$ be a Cartesian closed tangent category. We say that $\Cscr$ is a representable tangent category if there exists an object $D$ of $\Cscr$ for which there is a natural isomorphism $T(-) \cong [D,-]$.
\end{definition}

Because of this we will first explain how the $\Ind$-construction preserves adjoints.
\begin{proposition}\label{Prop: Ind preserves adjoints}
Let $F \dashv G:\Cscr \to \Dscr$ be an adjunction. Then there is an adjunction $\Ind(F) \dashv \Ind(G):\Ind(\Cscr) \to \Ind(\Dscr)$.
\end{proposition}
\begin{proof}
Let $\ul{X} = (X_i)_{i \in I}$ be an object of $\Ind(\Cscr)$ and let $\ul{Y} = (Y_j)_{j \in J}$ be an object of $\Ind(\Dscr)$. Then because $F \dashv G$, $\Cscr(X_i,GY_j) \cong \Dscr(FX_i,Y_j)$ for all $i \in I_0$ and all $j \in J_0$. Then we get that
\begin{align*}
\Ind(\Cscr)(\ul{X},\Ind(G)\ul{Y}) &= \Ind(\Cscr)((X_i),(GY_j)) = \lim_{\substack{\longleftarrow \\ i \in I}}\left(\lim_{\substack{\longrightarrow \\ j \in J}}\Cscr(X_i,GY_j)\right) \\
&\cong \lim_{\substack{\longleftarrow \\ i \in I}}\left(\lim_{\substack{\longrightarrow \\ j \in J}}\Dscr(FX_i,Y_j)\right) = \Ind(\Dscr)((FX_i),(Y_j)) = \Ind(\Dscr)(\Ind(F)\ul{X},\ul{Y}).
\end{align*}
Thus $\Ind(F) \dashv \Ind(G):\Ind(\Cscr) \to \Ind(\Dscr)$.
\end{proof}

We can then use this to expand internal hom functors on Cartesian closed categories to the $\Ind$-category at least when the co-representing object $\ul{X}$ is constant, i.e., when $\ul{X} = (X)_{\ast \in \lbrace \ast \rbrace}$.

\begin{proposition}\label{Prop: Reple internal hom of inding things}
Let $\Cscr$ be a for which there is an object $X \in \Cscr_0$ such that $\Cscr$ admits a product functor $(-) \times X:\Cscr \to \Cscr$ and there is an internal hom functor $[X,-]:\Cscr \to \Cscr$ with $(-) \times X \dashv [X,-]$. Then if $\ul{X} = (X)_{\ast \in \lbrace \ast \rbrace}$ is the constant object at $X$ in $\Ind(\Cscr)$ the internal hom functor $[\ul{X},-]:\Ind(\Cscr) \to \Ind(\Cscr)$ exists and is right adjoint to the product functor $(-) \times \ul{X}$.
\end{proposition}
\begin{proof}
That $\Ind((-) \times X) \dashv \Ind([X,-])$ is a consequence of Proposition \ref{Prop: Ind preserves adjoints}. Because there is a natural isomorphism of functors
\[
\Ind((-) \times X) \cong (-) \times \ul{X}
\]
we get that $(-) \times \ul{X} \dashv \Ind([X,-])$ and so the internal hom exists. The final statement about the functor $[\ul{X},-]$ being isomorphic to $\Ind([X,-])$ follows from the fact that for any $(Y_i)_{i \in I}$ and any $(Z_j)_{j \in J}$ in $\Ind(\Cscr)$,
\begin{align*}
\Ind(\Cscr)\big((Y_i), \left[\ul{X},(Z_j)\right]\big) &= \lim_{\substack{\longleftarrow \\ i \in I}}\left(\lim_{\substack{\longrightarrow \\ j \in J}}\Cscr(Y_i,[X,Z_j])\right) \cong \lim_{\substack{\longleftarrow \\ i \in I}}\left(\lim_{\substack{\longrightarrow \\ j \in J}}\Cscr(Y_i \times X, Z_j)\right) \\
&= \Ind(\Cscr)\left((Y_j) \times \ul{X}, (Z_j)\right)
\end{align*}
which gives the desired isomorphism $[\ul{X},-] \cong \Ind([X,-]).$
\end{proof}
\begin{proposition}\label{Prop: Ind of reple is reple}
Let $(\Cscr,\Tbb)$ be a representable tangent category. Then there is an object $\ul{D}$ in the $\Ind$-tangent category $(\Ind(\Cscr),\Ind(\Tbb))$ for which $\Ind(T) \cong [\ul{D},-]$.
\end{proposition}
\begin{proof}
Since $\Cscr$ is a representable tangent category, $\Cscr$ is Cartesian closed and there is an object $D \in \Cscr_0$ for which $[D,-] \cong T$. It then follows from Proposition \ref{Prop: Reple internal hom of inding things} that the internal hom functor $\Ind([D,-]) \cong [\ul{D},-]$ exists. A routine calculation shows
\[
[\ul{D},-] \cong \Ind([D,-]) \cong \Ind(T)
\]
which in turn proves the proposition.
\end{proof}
\begin{remark}
Proposition \ref{Prop: Ind of reple is reple} technically holds in more general (formal) situations than merely for Cartesian closed categories $\Cscr$. If there were a tangent category $(\Cscr,\Tbb)$ with an object $D \in \Cscr_0$ for which there is an internal hom functor and natural isomorphism $[D,-] \cong T$, then there is a natural isomorphism $[\ul{D},-] \cong \Ind(T)$ as well. I am unaware, however, of tangent categories $(\Cscr,\Tbb)$ with this property that are not Cartesian closed\footnote{And in fact I would be very interested in seeing such an example which arises ``in nature'' as it were.}.
\end{remark}

\subsection{Differential Bundles and Ind-Tangent Categories}
One of the most important structures we can encounter in the theory of tangent cateogries is that of differential bundles. These describe objects $E$ over some base object $M$ together with a structure map $q:E \to M$ that make $E$ look ``locally like'' they are given by tangent bundles. While these are of course related to vector bundles in algebraic geometry and differential geometry, they are also related to linear logic and allow a clean description of Cartesian differential categories (cf.\@ \cite[Section 3.5]{GeoffRobinBundle}) In this subsection we will study how the $\Ind$-construction translates pseudofunctors and describe the cases where we can say that an object $\ul{E} \to \ul{M}$ in $\Ind(\Cscr)$ are ``locally'' determined by differential bundles in $\Cscr$. Consequently we necessitate a short review of differential bundles and differential objects in tangent categories.

\begin{definition}[{\cite[Definition 2.1, 2.2]{GeoffRobinDiffStruct}}]
Let $\Cscr$ be a category. An additive bundle in $\Cscr$ is a morphism $p:A \to M$ such that $p:A \to M$ is a commutative monoid in $\Cscr_{/M}$. 

A morphism of additive bundles $(p:A \to M, \alpha:A \times_M A \to A, \zeta:M \to A)$ and $(q:B \to N, \mu:B \times_N B \to B, \eta:N \to B)$ is a pair of maps $f:A \to B$ and $g:M \to N$ such that the diagrams
\[
\begin{tikzcd}
A \ar[r]{}{f} \ar[d, swap]{}{p} & B \ar[d]{}{q} \\
M \ar[r, swap]{}{g} & N
\end{tikzcd}\qquad
\begin{tikzcd}
A \times_M A \ar[rr]{}{\langle f \circ \pi_0, f \circ \pi_1 \rangle} \ar[d, swap]{}{\alpha} & & B \times_N B \ar[d]{}{\mu} \\
A \ar[rr, swap]{}{f} & & B
\end{tikzcd}\qquad
\begin{tikzcd}
M \ar[r]{}{g} \ar[d, swap]{}{\zeta} & N \ar[d]{}{\eta} \\
A \ar[r, swap]{}{f} & B
\end{tikzcd}
\]
commute.
\end{definition}
\begin{definition}[{\cite[Definition 2.3]{GeoffRobinBundle}}]
Let $(\Cscr, \Tbb)$ be a tangent category. A differential bundle in $\Cscr$ is a quadruple $q = (q, \sigma, \zeta, \lambda)$ where:
\begin{itemize}
	\item $q:E \to M$ is a morphism which is a commutative monoid internal to $\Cscr_{/M}$ with addition map $\sigma:E_2 \to E$ and unit map $\zeta:M \to E$;
	\item $\lambda:E \to TE$ is a morphism called the lift of $q$;
	\item For any $n \in \N$, the pullback powers $E_n = E \times_M \cdots \times_M E$ exist in $\Cscr$ for all $n \in \N$ and for each $m \in \N$ the functor $T^m$ preserves these pullbacks;
	\item The pair $(\lambda, 0):( q, \sigma, \zeta, \lambda) \to (Tq, T\sigma, T\zeta, T\lambda)$ is an additive bundle map;
	\item The pair $(\lambda, \zeta):(q,\sigma,\zeta, \lambda) \to (p, +, 0, \ell)$ is an additive bundle morphism;
	\item The lift $\lambda$ is universal in the sense that if $\mu$ is the morphism
	\[
	\begin{tikzcd}
	E_2 \ar[rr]{}{\mu} \ar[dr, swap]{}{\langle \lambda \circ \pi_0, 0 \circ \pi_1 \rangle} & & TE \\
	 & T(E_2) \ar[ur, swap]{}{T\sigma}
	\end{tikzcd}	
	\]
	then the diagram below is a pullback which is preserved by $T^n$ for all $n \in \N$:
	\[
	\begin{tikzcd}
	E_2 \ar[r]{}{\mu} \ar[d, swap]{}{q \circ \pi_1} & TE \ar[d]{}{Tq} \\
	M \ar[r, swap]{}{0} & TM
	\end{tikzcd}
	\]
	\item The equation $\ell \circ \lambda = T(\lambda) \circ \lambda$ holds.
\end{itemize}
\end{definition}
\begin{definition}[{\cite[Definition 2.3]{GeoffRobinBundle}}]
A morphism of differential bundles from $(q:E \to M, \sigma, \zeta, \lambda)$ to $(q^{\prime}:F \to N, \mu, \eta, \rho)$ is a pair of morphisms $f:E \to F$ and $g:M \to N$ such that the diagram
\[
\begin{tikzcd}
E \ar[r]{}{f} \ar[d, swap]{}{q} & F \ar[d]{}{q^{\prime}} \\
M \ar[r, swap]{}{g} & N
\end{tikzcd}
\]
commutes. Additionally, we say that the morphism $(f,g)$ is linear if it preserves the lift in the sense that the diagram
\[
\begin{tikzcd}
E \ar[r]{}{f} \ar[d, swap]{}{\lambda} & F \ar[d]{}{\rho} \\
T(E) \ar[r, swap]{}{Tf} & T(F)
\end{tikzcd}
\]
commutes.
\end{definition}
\begin{remark}
In the definition of a morphism of differential bundles, the pair $(f,g)$ need not be an additive bundle map, i.e., $(f,g)$ need not preserve the monoidal operations of $q:A \to M$. However, by \cite[Proposition 2.16]{GeoffRobinBundle} it does follow that if the bundle map is \emph{linear} then it is also an additive bundle map.
\end{remark}

We now show the first preservation result of the subsection: namely that two ind-objects indexed by the same filtered category which are locally differential bundles with linear transition maps give rise to an ind-differential bundle.

\begin{proposition}\label{Prop: Diff Bunlde in Ind}
Let $I$ be a filtered index category and let $\Cscr$ be a tangent category. If $\ul{E} = (E_i)_{i \in I}$ and $\ul{M} = (M_i)_{i \in I}$ are $\Ind$-objects such that:
\begin{itemize}
	\item For each $i \in I_0$, $E_i$ and $M_i$ there is a differential bundle $(q_i:E_i \to M_i, \sigma_i:E_i \times_{M_i} E_i \to E_i, \zeta_i:M_i \to E_i, \lambda_i:E_i \to TE_i)$;
	\item For each map $\varphi:i \to i^{\prime}$ in $I$, if $\tilde{\varphi}:E_i \to E_{i^{\prime}}$ and $\widehat{\varphi}:M_i \to M_{i^{\prime}}$ denote the corresponding structure maps in $\ul{E}$ and $\ul{M}$ then $(\tilde{\varphi},\widehat{\varphi}):(q_i,\sigma_i, \zeta_i, \lambda_i) \to (q_{i^{\prime}}, \sigma_{i^{\prime}}, \zeta_{i^{\prime}}, \lambda_{i^{\prime}})$ is a linear morphism of differential bundles;
\end{itemize}
then there is a differential bundle $(\ul{q}, \ul{\sigma}, \ul{\zeta}, \ul{\lambda})$ in $\Ind(\Cscr)$.
\end{proposition}
\begin{proof}
First note that the condition that there be a bundle map $q_i:E_i \to M_i$ for all $i \in I$ and that the maps $(\tilde{\varphi}, \widehat{\varphi})$ being a linear differential bundle map mean that the diagrams
\[
\begin{tikzcd}
E_i \ar[r]{}{\tilde{\varphi}} \ar[d, swap]{}{q_i} & E_{i^{\prime}} \ar[d]{}{q_{i^{\prime}}} \\
M_i \ar[r, swap]{}{\widehat{\varphi}} & M_{i^{\prime}}
\end{tikzcd}\qquad
\begin{tikzcd}
E_i \ar[r]{}{\tilde{\varphi}} \ar[d, swap]{}{\lambda_i} & E_{i^{\prime}} \ar[d]{}{\lambda_{i^{\prime}}} \\
TE_i \ar[r, swap]{}{T\tilde{\varphi}} & TE_{i^{\prime}}
\end{tikzcd}
\]
all commute for every morphism $\varphi:i \to i^{\prime}$ in $I$. This allows us to deduce that there are maps $(q_i)_{i \in I} =: \ul{E} \to \ul{M}$ and $(\lambda_i)_{i \in I} =: \ul{\lambda}:\ul{E} \to \Ind(T)\ul{E}$ in $\Ind(\Cscr)$. Similarly, by \cite[Proposition 2.16]{GeoffRobinBundle} it follows that each pair $(\tilde{\varphi},\widehat{\varphi})$ is an additive morphism of bundles which implies that the diagrams
\[
\begin{tikzcd}
	E_i \times_{M_i} E_i \ar[rr]{}{\langle \tilde{\varphi} \circ \pi_0, \tilde{\varphi} \circ \pi_1\rangle} \ar[d, swap]{}{\sigma_i} & & E_{i^{\prime}} \times_{M_{i^{\prime}}} E_{i^{\prime}} \ar[d]{}{\sigma_{i^{\prime}}} \\
	E_i \ar[rr, swap]{}{\tilde{\varphi}} & & E_{i^{\prime}}
\end{tikzcd}\qquad
\begin{tikzcd}
	M_i \ar[r]{}{\widehat{\varphi}} \ar[d, swap]{}{\zeta_i} & M_{i^{\prime}} \ar[d]{}{\zeta_{i^{\prime}}} \\
	E_i \ar[r, swap]{}{\tilde{\varphi}} & E_{i^{\prime}}
\end{tikzcd}
\]
commute. Since $\ul{E} \times_{\ul{M}} \ul{E} = (E_i \times_{M_i} E_i)_{i \in I}$, this allows us to deduce that the maps $(\sigma_i)_{i \in I} =: \ul{\sigma}:\ul{E} \times_{\ul{M}} \ul{E} \to \ul{E}$ and $(\zeta_i)_{i \in I} =: \ul{\zeta}:\ul{M} \to \ul{E}$ exist in $\Ind(\Cscr)$ as well. This defines our quadruple $(\ul{q}, \ul{\sigma}, \ul{\zeta}, \ul{\lambda})$ which we will now show is a differential bundle in an axiom-by-axiom verification.

First note that the pullbacks $\ul{E}_n := \ul{E} \times_{\ul{M}} \cdots \times_{\ul{M}} \ul{E}$ exist by an argument mutatis mutandis to the argument which showed why tangent pullbacks exist in Proposition \ref{Prop: Ind of Tangent pullback is tangent pullback of ind}, i.e., the object $\ul{X} = (E_i \times_{M_i} \times_{M_i} \cdots \times_{M_i} E_i)_{i \in I}$ represents the pullback $\ul{E}_n$. The argument as to why the compositional powers $\Ind(T)^m$ preserve these pullbacks is similar to the argument why $\Ind(T)^m$ preserves tangent pullbacks given in Corollary \ref{Cor: Ind Tangent functor commutes with Ind tangent pullbacks}.

To show that $(\ul{\lambda},\Ind(0)_{\ul{M}}):(\ul{q},\ul{\sigma},\ul{\zeta},\ul{\lambda}) \to (\Ind(T)\ul{q}, \Ind(T)(\ul{\sigma}), \Ind(T)(\ul{\zeta}), \Ind(T)(\ul{\lambda}))$ is an additive bundle morphism we need to show that the diagrams
\[
\begin{tikzcd}
\ul{E} \ar[r]{}{\ul{\lambda}} \ar[d, swap]{}{\ul{q}} & \Ind(T)\ul{E} \ar[d]{}{\Ind(T)\ul{q}} \\
\ul{M} \ar[r, swap]{}{\Ind(0)_{\ul{M}}} & \Ind(T)\ul{M}
\end{tikzcd}\qquad
\begin{tikzcd}
\ul{E} \times_{\ul{M}} \ul{E} \ar[rr]{}{\langle \ul{\lambda} \circ \pi_0, \ul{\lambda}\circ\pi_1 \rangle} \ar[d, swap]{}{\ul{\sigma}} & & \Ind(T)\ul{E} \times_{\Ind(T)\ul{M}} \Ind(T)\ul{E} \ar[d]{}{\Ind(T)\ul{\sigma}} \\
\ul{E} \ar[rr]{}{\ul{\lambda}} & & \Ind(T)\ul{E}
\end{tikzcd}\qquad
\begin{tikzcd}
\ul{M} \ar[r]{}{\Ind(0)_{\ul{M}}} \ar[d, swap]{}{\ul{\zeta}} & \Ind(T)\ul{M} \ar[d]{}{\Ind(T)\ul{\zeta}} \\
\ul{E} \ar[r, swap]{}{\ul{\lambda}} & \Ind(T)\ul{E}
\end{tikzcd}
\]
all commute. However, as each of the diagrams
\[
\begin{tikzcd}
	{E}_i \ar[r]{}{\ul{\lambda}} \ar[d, swap]{}{\ul{q}} & TE_i \ar[d]{}{Tq_i} \\
	M_i \ar[r, swap]{}{0_{{M}_{i}}} & TM_i
\end{tikzcd}\qquad
\begin{tikzcd}
	E_i \times_{M_i} E_i \ar[rr]{}{\langle \lambda_i \circ \pi_0, \lambda_i\circ\pi_1 \rangle} \ar[d, swap]{}{\sigma_i} & & TE_i \times_{TM_i} TE_i \ar[d]{}{T\sigma_i} \\
	E_i \ar[rr]{}{\lambda_i} & & TE_i
\end{tikzcd}\qquad
\begin{tikzcd}
	M_i \ar[r]{}{0_{M_i}} \ar[d, swap]{}{{\zeta}_i} & TM_i \ar[d]{}{T{\zeta}_i} \\
	E_i \ar[r, swap]{}{{\lambda}_i} & TE_i
\end{tikzcd}
\]
commutes for each $i \in I_0$, and each given pullback represents the $i$-component of the corresponding $\Ind$-object it follows that the diagrams between $\Ind$-objects commute. Thus $(\ul{\lambda},\Ind(0)_{\ul{M}})$ is an additive bundle map. Establishing that $(\ul{\lambda},\ul{\zeta})$ is an additive bundle morphism is done similarly and hence omitted.

For the universality of the lift $\ul{\lambda}$ we note that each of the maps in the desired pullback are defined $I$-locally. Consequently we can deduce that the pullbacks exist and are preserved by all powers $\Ind(T)^m$ follows similarly to our verification of the existence and preservation of $\ul{E}_n$.

Finally we verify the last equation holds. Explicitly we calculate that
\begin{align*}
\hat{\ell} \circ \ul{\lambda} &= \phi_{T,T} \circ (\ell_i)_{i \in I} \circ (\lambda_{i \in I}) = \phi_{T,T} \circ (\ell_i \circ \lambda_i)_{i \in I} = \phi_{T,T} \circ (T(\lambda_i) \circ \lambda_i)_{i \in I} = \phi_{T,T} \circ \phi_{T,T}^{-1} \circ \Ind(T)\ul{\lambda} \circ \ul{\lambda} \\
&= \Ind(T)\ul{\lambda} \circ \ul{\lambda}
\end{align*}
since the equation $\ell_i \circ \lambda_i = T(\lambda_i) \circ \lambda_i$ holds for all $i \in I_0$. Thus $(\ul{q}, \ul{\sigma}, \ul{\zeta}, \ul{\lambda})$ is a differential bundle in $\Ind(\Cscr)$.
\end{proof}

A particularly important class of differential bundles in a tangent category with finite limits are the bundles over the terminal object $\top$. These objects have various remarkable properties, such as satisfying $TE \cong E \times E$ among other relations (cf.\@ \cite[Definition 3.1, Proposition 3.4]{GeoffRobinBundle} --- the definition describes the definition of differential structures on objects and the proposition proves that in finitely complete tangent categories such structures are exactly differential bundles whose bundle map is a map to the terminal object). We will explore these more afterwards when we examine the structure of $\Ind(\Cscr)$ in the case where $\Cscr$ is a Cartesian differential category, but for the moment we simply have a nice corollary to consider. We will, however, make a formal declared definition of differential objects for easier reading and then give the corollary.
\begin{definition}
Let $\Cscr$ be a finitely complete tangent category. A differential object in $\Cscr$ is an object $X$ of $\Cscr$ such that if $\top$ is the terminal object of $\Cscr$ then the unique map $!_X:X \to \top$ is a differential bundle in $\Cscr$.
\end{definition}
\begin{definition}\label{Defn: Diff}
Let $\Cscr$ be a tangent category with finite products. We denote by $\mathbf{Diff}(\Cscr)$ the subcategory of $\Cscr$ generated by the differential bundles over $\top_{\Cscr}$ and the linear bundle maps $(\varphi,\id_{\top})$ between them.
\end{definition}
\begin{corollary}\label{Cor: Diff Bundle over Fin Comp Category}
Let $\Cscr$ be a tangent category with finite products and let $I$ be a filtered index category. Then if $\ul{E} = (E_i)_{i \in I}$ is an $\Ind$-object such that each $E_i$ is a differential bundle over the terminal object $\top_{\Cscr}$ and if each structure map $E_i \to E_{i^{\prime}}$ is part of a linear morphism of differential bundles then $\ul{E}$ is a differential object in $\Ind(\Cscr)$.
\end{corollary}
\begin{proof}
Begin by observing that the object $\ul{\top} = (\top_{\Cscr})_{i \in I}$ is an $\Ind$-object. The hypotheses in the statement of the corollary together with Proposition \ref{Prop: Diff Bunlde in Ind} give that the map $\ul{!}:\ul{E} \to \ul{\top}$ is a differential bundle in $\Ind(\Cscr)$. The result now follows from \cite[Proposition 3.4]{GeoffRobinBundle} because
\[
\ul{\top} \cong \top_{\Ind(\Cscr)}
\]
so $\ul{E}$ is a differential bundle over the terminal object of $\Ind(\Cscr)$.
\end{proof}

Let us now discuss some of the $\Ind$-tangent categorical structure of a Cartesian differential category. Cartesian differential categories (CDCs) are important objects in the pantheon of categorical differential geometry and have their genesis in providing semantics for the differential $\lambda$-calculus and for many of the differential operations performed in algebraic and differntial geometry; cf.\@ \cite{RickRobinRobertCDC}. We will recall these categories below, but a main feature they enjoy is that every object in a CDC is a differential object, and that is the main property we will be focusing on and working with. To describe CDCs, however, we'll have to also first introduce left additive categories and a small barrage of terminology surrounding these objects, which are essentially categories \emph{almost} enriched in the category of modules over a rig save that only post-composition is a morphism of modules. These are examples of what are called skew-enriched categories in \cite{StreetSkewClosed}.

\begin{definition}[{\cite[Definition 1.1.1]{RickRobinRobertCDC}; \cite[Definition 2.1]{JSproperties}}]
Let $A$ be a crig\footnote{That is, $A$ is a commutative ring without negatives.}. A left $A$-linear category is a category $\Cscr$ such that:
\begin{itemize}
	\item For every pair of objects $X, Y \in \Cscr_0$ the hom-set $\Cscr(X,Y)$ is an $A$-module. We will write the juxtaposition $af$ of an element $a \in A$ with a morphism $f \in \Cscr(X,Y)$ to denote the action of $A$ on $\Cscr(X,Y)$.
	\item For any morphism $f:X \to X^{\prime}$ in $\Cscr$ and for any $B \in \Cscr_0$ the pre-composition morphism $f^{\ast}:\Cscr(X^{\prime},Y) \to \Cscr(X,Y)$ is a morphism of $A$-modules. That is, for all $a, b \in A$ and for all $g, h \in \Cscr(X^{\prime}, Y)$, the equation
	\[
	(ag + bh) \circ f = a(g \circ f) + b(h \circ f)
	\]
	holds.
\end{itemize}
\end{definition}
\begin{definition}[{\cite[Definition 1.1.1]{RickRobinRobertCDC}; \cite[Definition 2.1]{JSproperties}}]
Let $\Cscr$ be a left $A$-linear category for a crig $A$. Then we say that a morphism $f \in \Cscr_1$ is linear if the post-composition by $f$ map is also an $A$-module morphism. That is, if $f:X \to X^{\prime}$, $Y \in \Cscr_0$ is arbitrary, and if $a, b \in A$ and $g, h \in \Cscr(Y,X)$ then the equation
\[
h \circ (af + bg) = a(f \circ g) + b(f \circ g)
\]
holds in $\Cscr(Y, X^{\prime})$.
\end{definition}
\begin{definition}[{\cite[Definition 1.1.1]{RickRobinRobertCDC}; \cite[Definition 2.1]{JSproperties}}]
Let $A$ be a crig and let $\Cscr$ be a left $A$-linear category. The subcategory of $\Cscr$ spanned by the linear morphisms is denoted $\Cscr_{\lin}$. Moreover, a left $A$-linear category is said to be $A$-linear if and only if $\Cscr = \Cscr_{\lin}$.
\end{definition}

With this we can describe Cartesian $A$-linear categories and then, finally, Cartesian Differential Categories.
\begin{definition}[{\cite[Definition 1.2.1]{RickRobinRobertCDC}; \cite[Definition 2.2]{JSproperties}}]
Let $A$ be a crig. A Cartesian left $A$-linear category is a left $A$-linear category $\Cscr$ with finite products for which all the projection maps
\[
\pi_{i}^{0\cdots n}:A_0 \times A_1 \times \cdots \times A_n \to A_i
\]
are linear.
\end{definition}
We now get to the definition of a CDC. It is worth noting, however, that we are taking the convention used in \cite{JSproperties} which says that the linear argument of the differential operator is the second argument; earlier works, such as \cite{RickRobinRobertCDC}, used the convention that the first argument was the linear argument. This will not present any serious issue, but it is worth noting the swap which appears in the various references in the literature.
\begin{definition}[{\cite[Definition 2.1.1]{RickRobinRobertCDC}; \cite[Definition 2.5]{JSproperties}}]
Let $A$ be a crig. A Cartesian differential $A$-linear category is a Cartesian left $A$-linear category $\Cscr$ equipped with a differential combinator $D$\footnote{So $D$ is a function on hom-sets of the form $D:\Cscr(X,Y) \to \Cscr(X \times X, Y)$.} which gives the differentiation of a morphism,
\begin{prooftree}
\AxiomC{$f:X \to Y$}
\UnaryInfC{$D(f):X \times X \to Y$}
\end{prooftree}
where $D(f)$ is called the derivative of $f$, which satisfies the following seven axioms:
\begin{enumerate}
	\item[CD1] $D(af + bg) = aD(f) + bD(g)$ for all $a, b \in A$ and for all morphisms $f, g$.
	\item[CD2] $D(f)\circ \langle g, ah + bk \rangle = a(D(f) \circ \langle g,  h\rangle) + b(D(g)\circ \langle g, k \rangle)$ for all $a, b \in A$.
	\item[CD3] $D(\id_X) = \pi_1^{X,X}:X \times X \to X$ and, for any product $X_0 \times \cdots \times X_n$,
	\[
	D(\pi_{i}^{0\cdots n}) = \pi_{i}^{0\cdots n} \circ \pi_1^{X_0\times \cdots \times X_n, X_0 \times \cdots \times X_n} = \pi_{n+i+1}^{(0\cdots n)(0\cdots n)}:(X_0 \times \cdots \times X_n) \times (X_0 \times \cdots \times X_n) \to X_i.
	\]
	\item[CD4] $D\big(\langle f_0, \cdots, f_n\rangle\big) = \langle D(f_0), \cdots, D(f_n) \rangle$.
	\item[CD5] $D(g \circ f) = D(g) \circ \langle f \circ \pi_0, D(f) \rangle$.
	\item[CD6] $D(D(f)) \circ \langle g, h, 0, k\rangle = D(f) \circ \langle g, k \rangle$.
	\item[CD7] $D(D(f)) \circ \langle g, h, k, 0 \rangle = D(D(f)) \circ \langle g, k, h, 0 \rangle.$
\end{enumerate}
\end{definition}
When $\Cscr$ is an $A$-linear CDC then there is a class of maps in $\Cscr$ which are more important than those which are merely linear: the maps which are linear with respect to the differential combinator $D$. These maps are called differential-linear and, by a result in \cite{GeoffRobinBundle}, correspond to linear bundle maps for differential bundles over a terminal object.

\begin{definition}[{\cite[Definition 2.6]{JSproperties}}]
Let $A$ be a crig and let $\Cscr$ be an $A$-linear CDC. We say that a map $f:A \to B$ is differntial-linear (in short form, we say $f$ is $D$-linear) if and only if $D(f) = f \circ \pi_0$. The subcategory of $\Cscr$ spanned by the $D$-linear maps is denoted $\Cscr_{\Dlin}$.
\end{definition}
It is, of course, a fact that $\Cscr_{\Dlin}$ is a subcategory of $\Cscr_{\lin}$ by \cite[Lemma 2.2]{RickRobinRobertCDC}. However, it need not be the case that every $A$-linear morphism in a CDC be $D$-linear; thus we restrict our attention primarily to the $D$-linear notion of linearity as these are the maps which interact with the tangent categorical properties best.

Here are the important results for our purposes that show $A$-linear CDCs are tangent categories where every object is a differential bundle over a terminal object $\top$ and that linear bundle morphisms between differential objects over $\top$ in an $A$-linear CDC are $D$-linear maps with respect to their standard tangent structure.

\begin{proposition}[{\cite[Section 3.4]{GeoffRobinBundle}}]\label{Prop: CDC is tan cat and all obs are diff bundles over 1}
Let $A$ be a crig and let $\Cscr$ be an $A$-linear CDC. Then $\Cscr$ is a tangent category with tangent functor defined on obejcts by
\[
T(A) := A \times A
\]
and defined on morphisms by
\[
T(f) := \langle D(f), f \circ \pi_1 \rangle.
\]
In particular, every object in a CDC equipped with this tangent structure is a differential bundle over a terminal object.
\end{proposition}
\begin{proposition}[{\cite[Section 3.4]{GeoffRobinBundle}}]\label{Prop: Dlinear is bundle map linear}
Let $A$ be a crig and let $\Cscr$ be an $A$-linear CDC. Then $f:A \to B$ is a linear morphism of differential bundles if and only if $f$ is $D$-linear.
\end{proposition}
\begin{proof}
The verification of the $\implies$ direction of the proof is given in \cite[Section 3.4]{GeoffRobinBundle} while the $\impliedby$ direction follows from a routine verification.
\end{proof}

Our first result regarding the $\Ind$-category of a CDC and its $\Ind$-tangent structure shows that we can actually use the $\Ind$-category of a CDC to classify if $\Cscr = \Cscr_{\Dlin}$.

\begin{proposition}\label{Prop: Ind of CDC has all diff objs means C is Dlinear}
Let $A$ be a crig and let $\Cscr$ be an $A$-linear CDC. Then $\Cscr = \Cscr_{\Dlin}$ if and only if every object in $\Ind(\Cscr)$ is a differential bundle over the terminal object $\top_{\Ind(\Cscr)}$.
\end{proposition}
\begin{proof}
$\implies$: Assume that $\Cscr = \Cscr_{\Dlin}$. Then by Propositions \ref{Prop: CDC is tan cat and all obs are diff bundles over 1} and \ref{Prop: Dlinear is bundle map linear} $\Cscr$ is a tangent category, every object in $\Cscr$ is a differential bundle over $\top_{\Cscr}$, and every morphism is a linear map of differential bundles. We can thus apply Corollary \ref{Cor: Diff Bundle over Fin Comp Category} to deduce that every $\Ind$-object of $\Cscr$ is a differential bundle over $\top_{\Ind(\Cscr)}$.

$\impliedby$: Assume that every object in $\Ind(\Cscr)$ is a differential bundle over the terminal object $\top_{\Ind(\Cscr)}$. Consider the category $\mathbbm{2}$, which we display below:
\[
\begin{tikzcd}
0 \ar[loop left]{}{\id_0} \ar[r]{}{\varphi} & 1 \ar[loop right]{}{\id_1}	
\end{tikzcd}
\]
Then $\mathbbm{2}$ is a finite (and hence small) filtered category so any functor $F:\mathbbm{2} \to \Cscr$ determines an $\Ind$-object of $\Cscr$. Let $F$ be any such functor and let $\ul{\top}:\mathbbm{2}\to \Cscr$ be the functor with $\ul{\top}(0) = \top_{\Cscr} = \ul{\top}(1).$ Then $\ul{\top} \cong \top_{\Ind(\Cscr)}$ so the functor $F$ is a differential object over $\ul{\top}$ by construction. However, it then follows that $F(0)$ and $F(1)$ are differential bundles over $\top$ in $\Cscr$ and the pair $(F(\varphi), \ul{\top}(\varphi))$ is a linear morphism of differential bundles. However, this implies that $F(\varphi)$ is a $D$-linear morphism by Proposition \ref{Prop: Dlinear is bundle map linear}, so using that there is a canonical bijection of sets
\[
[\mathbbm{2},\Cscr]_0 \cong \Cscr_1
\]
allows us to deduce that every morphism in $\Cscr$ is $D$-linear.
\end{proof}

A corollary of this construction is that it allows us to give a necessary and sufficient condition for recongnizing differential bundles in Ind-categories of CDCs.
\begin{corollary}\label{Cor: Dlin nes}
Let $\Cscr$ be an $A$-linear CDC. In order for $F:I \to \Cscr$ to be a differential bundle in $\Ind(\Cscr)$ it is necessary and sufficient that for $\varphi \in I_1$, every morphism $F(\varphi)$ is $D$-linear in $\Cscr$.
\end{corollary}

This means that when we study the $\Ind$-category of a CDC in order to determine when $\Ind(\Cscr)$ is \emph{not} a CDC we can use the following recipe:
\begin{enumerate}
	\item Find two objects of our CDC, $X$ and $Y$, such that there is a non-$D$-linear morphism $\alpha:X \to Y$.
	\item Consider the $\Ind$-object $F:\mathbbm{2} \to \Cscr$ given by $F(0) = X, F(1) = Y, F(\alpha) = \alpha$.
	\item Note that $F$ is not a differential bundle over $\top_{\Ind(\Cscr)}$ and so is not a differential object.
	\item Conclude that $\Ind(\Cscr)$ is not a CDC because not every object is a differential object.
\end{enumerate}
This leads to a straightforward way of recognizing when $\Ind$-categories of CDCs are not themselves CDCs.

\begin{proposition}\label{Prop: Ind of CDC is not a CDC}
Let $A$ be a crig and let $\Cscr$ be an $A$-linear CDC. Then if $\Cscr_{\Dlin} \ne \Cscr$, $\Ind(\Cscr)$ is not a CDC.
\end{proposition}

\begin{remark}
It remains open to determine exactly when the $\Ind$-category of a CDC is itself a CDC. In light of Propositions \ref{Prop: Ind of CDC has all diff objs means C is Dlinear} and \ref{Prop: Ind of CDC is not a CDC} if $\Cscr$ is a CDC in order for $\Ind(\Cscr)$ to be a CDC we require that $\Cscr = \Cscr_{\Dlin}$. However, to make $\Ind(\Cscr)$ a CDC we need the differential combinators to behave suitably well with filtered colimits.
\end{remark}

\section{Ind-Tangent Category Examples}
In this section we will give some computations regarding the Ind-tangent category of some tangent categories of interest. What is of particular note is that the examples here give various different ways of discussing infinite-dimensional differential-geometric information: $\Ind(\mathbf{CAlg}_A)$ and $\Ind(\APoly)$ both give algebraic (ring-or-rig-theoretic) analogues of infinite-dimensional differential algebra; $\Ind(\Sch_{/S})$ gives differential insight into formal\footnote{In the sense of the theory of formal schemes and formal geometry in the spirit of the the Existence Theorem for formal functions (cf.\@ \cite[The{\'e}or{\`e}me 4.1.5]{EGA3}).} infinite-dimensional Euclidean manifolds.

\subsection{The Tangent Structure on Commutative $A$-Algebras}
Let $A$ be a cring\footnote{A commutative ring with identity.}. The tangent structure on the category of commutative $A$-algebras takes the form
\[
T(R) := R[\epsilon] \cong \frac{R[x]}{(x^2)} \cong R \otimes_A \frac{A[x]}{(x^2)}
\]
and with morphisms adapted correspondingly. To give a concrete description of the tangent functor $\Ind(\mathbf{CAlg}_{A})$ we will use the functorial description of $\Ind$. Let $I$ be a fitlered category and let $F:I \to \mathbf{CAlg}_A$ be a functor. Then $\Ind(T)(F) = T \circ F$ so we can describe the $\Ind$-object of $\mathbf{CAlg}_A$ as follows:
\begin{itemize}
	\item For any $i \in I_0$, the object assignment of $\Ind(T)(F)$ is given by $(T \circ F)(i) = F(i)[\epsilon]$.
	\item For any morphism $\varphi:i \to i^{\prime}$ in $I$, the morphism assignment of $\Ind(T)(F)$ is given by $(\Ind(T)(F))(\varphi) = (T \circ F)(\varphi) = F(\varphi)[\epsilon]$ where $F(\varphi)[\epsilon]$ is the unique morphism making the diagram
	\[
	\begin{tikzcd}
	F(i)[\epsilon] \ar[r, dashed]{}{\exists!\,F(\varphi)[\epsilon]} \ar[d, swap]{}{\cong} & F(i^{\prime})[\epsilon] \ar[d]{}{\cong} \\
	\frac{F(i)[x]}{(x^2)} \ar[d, swap]{}{\cong} \ar[r]{}{r \mapsto F(\varphi)(r)} \ar[r, swap]{}{x \mapsto x} & \frac{F(i^{\prime})[x]}{(x^2)} \ar[d]{}{\cong} \\
	F(i) \otimes_A \frac{A[x]}{(x^2)} \ar[r, swap]{}{F(\varphi) \otimes  \id} & F(i^{\prime}) \otimes_A \frac{A[x]}{(x^2)}
	\end{tikzcd}
	\]
	commute.
\end{itemize}
In particular this means that $\Ind(T)$ acts by taking a filtered diagram in $\mathbf{CAlg}_{A}$ and applies the $(-)[\epsilon]$ functor to it at every possible stage.

\subsection{The Zariski Tangent Structure on Formal Schemes}
In a recent work, \cite{GeoffJS}, Crutwell and LeMay have proved that the category $\Sch_{/S}$ admits a tangent structure for any base scheme $S$. The tangent functor $T:\Sch_{/S} \to \Sch_{/S}$ sends an $S$-scheme $X$ to the (Zariski) tangent fibre of $X$ relative to $S$ constructed by Grothendieck in \cite[Section 16.5]{EGA44}. In particular, for an $S$-scheme $X$ we have\footnote{The version of the tangent functor and maps constructed in \cite{GeoffTalk} is given for affine schemes, but because the $\Sym$ functor commutes with tensors on $\QCoh(X)$ and because the sheaf of differentials $\Omega_{X/S}^1$ is a quasi-coherent sheaf, everything regarding this functor may be checked Zariski-locally, i.e., affine-locally on both the target and the base. We'll describe this more in detail later, but it is worth remarking now.}
\[
T(X) := T_{X/S} = \Spec\left(\Sym(\Omega^1_{X/S})\right).
\]
For an affine scheme $S = \Spec A$ and an affine $S$-scheme $X = \Spec B$, the scheme $T_{X/S} = T_{B/A} = \Spec(\Sym(\Omega^1_{B/A}))$ is an affine scheme. The ring
\[
C = \Sym\left(\Omega_{B/A}^{1}\right)
\]
is generated by symbols $b, \mathrm{d}b$ for $b \in B$ generated by the rules that addition and multiplication for symbols from $b \in B$ are as in $B$ and the Leibniz rule
\[
\mathrm{d}(bb^{\prime}) = b^{\prime}\mathrm{d}(b) + b\mathrm{d}(b^{\prime})
\]
holds with $\mathrm{d}b\mathrm{d}b^{\prime} = 0$ and $\mathrm{d}(a) = 0$ for $a \in A$. With this definition we find that 
\[
T_2(X) = (T_{B/A})_{2} = \Spec\left(\Sym(\Omega_{B/A}^1) \otimes_B \Sym(\Omega_{B/A}^{1})\right)
\]
and that
\[
T^2(X) = T_{T_{B/A}/A} = \Sym\left(\Omega^{1}_{\Sym(\Omega_{B/A}^1)/A}\right).
\]
It is worth noting for what follows that the algebra $C$ with $T^2_{B/A} = \Spec C$,
\[
C = \Sym\left(\Omega^1_{\Sym(\Omega_{B/A}^1)}\right),
\] 
is generated by symbols $b, \mathrm{d}b, \delta b,$ and $\delta\mathrm{d}(b)$ for all $b \in B$. Essentially, there is a new derivational neighborhood $\delta$ of $\Sym(\Omega_{B/A}^1)$ which gives us a notion of $2$-jets and a distinct direction of $1$-jets (the $\delta$-direction).

\begin{Theorem}[\cite{GeoffTalk}]\label{Thm: Zariski tangent structure}
For affine schemes $\Spec B \to \Spec A$ the tangent structure $(\Sch_{/\Spec A}, \mathbb{T})$ is generated by the maps and functors on affine schemes:
\begin{itemize}
	\item The tangent functor is given by $T_{\Spec B/\Spec A} = T_{B/A} = \Spec\left(\Sym(\Omega_{B/A}^1)\right)$.
	\item The bundle map $p_B:T_{B/A} \to \Spec B$ is the spectrum of the ring map
	\[
	q_B:B \to \Sym(\Omega_{B/A}^1)
	\]
	generated by $b \mapsto b$.
	\item The zero map $0_B:T_{B/A} \to \Spec B$ is the spectrum of the ring map
	\[
	\zeta_B:Sym(\Omega_{B/A}^{1}) \to B
	\]
	given by $b \mapsto b, \mathrm{d}b \mapsto 0$. 
	\item The bundle addition map $+_B:(T_{B/A})_{2} \to T_{B/A}$ is the spectrum of the map
	\[
	\operatorname{add}_B:\Sym(\Omega^{1}_{B/A}) \to \Sym(\Omega^{1}_{B/A}) \otimes_B \Sym(\Omega^{1}_{B/A})
	\]
	given by $b \mapsto b \otimes 1_B, \mathrm{d}b \mapsto \mathrm{d}b \otimes 1 + 1 \otimes \mathrm{d}b$.
	\item The vertical lift $\ell_B: T_{B/A} \to T^2_{B/A}$ is given as the spectrum of the ring map
	\[
	v_B:\Sym\left(\Omega^1_{\Sym(\Omega_{B/A}^1)}\right) \to \Sym\left(\Omega^1_{B/A}\right)
	\]
	generated by $b \mapsto b, \mathrm{d}b \mapsto 0, \delta b \mapsto 0$, $\delta\mathrm{d}(b) \mapsto \mathrm{d}b$.
	\item The canoncial flip is the map $c_B:T^2_{B/A} \to T^2_{B/A}$ generated as the spectrum of the ring map
	\[
	\gamma_B:\Sym\left(\Omega^1_{\Sym(\Omega_{B/A}^1)}\right) \to \Sym\left(\Omega^1_{\Sym(\Omega_{B/A}^1)}\right)
	\]
	which interchanges $1$-jets, i.e., the map is generated by $b \mapsto b, \mathrm{d}b \mapsto \delta b, \delta b \mapsto \mathrm{d}b$ and $\delta\mathrm{d}b \mapsto \delta\mathrm{d}b$.
\end{itemize}
\end{Theorem}
\begin{definition}\label{Defn: Zariski tangent structure on Sch over X}
We define the Zariski tangent structure on a scheme $S$ to be the tangent structure $\Tbb_{\Zar{S}} = (T_{-/S}, p, 0, +, \ell, c)$ on the category $\Sch_{/S}$ described by Theorem \ref{Thm: Zariski tangent structure}.
\end{definition}
\begin{remark}
This tangent category and tangent structure is also used and studied in \cite{DoretteMe}. The Zariski tangent structure there is used to study a category of equivariant tangent schemes over a base $G$-variety $X$; we will not take this direction in this paper save for a short discussion at the end of this paper. It is worth noting, however, that in \cite{DoretteMe} the various functoriality and naturality conditions involving the Zariski Tangent Structure are established, and we will refer to those results when establishing the functoriality of the $\Ind$-tangent structure on schemes and the tangent structure on the category of formal schemes.
\end{remark}

This is, in some sense, the ``canonical'' tangent structure on $\Sch_{/S}$, as the tangent scheme $T_{X/S}$ captures $S$-derivations of $X$ in the following sense: If $S = \Spec K$ for a field $K$ and if $\Spec A$ is an affine $K$-scheme, then the $A$-points of the tangent scheme $T_{X/S}(A)$ satisfy
\[
T_{X/S}(A) = \Sch_{/K}(\Spec A,T_{X/S}) \cong \Sch_{/K}\left(\Spec\left(\frac{A[x]}{(x^2)}\right), X\right)
\]
so in particular $K$-points of $T_{X/S}$ give the $1$-differentials $\Omega_{X/K}^1$. Moreover, for any closed point $x \in \lvert X \rvert$ we have a canonical isomorphism
\[
T_{X/S}(x) \cong \KAlg\left(\frac{\mfrak_{x}}{\mfrak^{2}_{x}},K\right)
\]
of $T_{X/S}(x)$ with the Zariski tangent space of $X$ over $K$.

\begin{proposition}[\cite{DoretteMe}]\label{Prop: Functoriality of Scheme Morphisms}
Let $S$ be a scheme and let $f:X \to Y$ be a morphism of $S$-schemes. Then the pullback functor $f^{\ast}:\Sch_{/Y} \to \Sch_{/X}$ is part of a strong tangent morphism $(f^{\ast}, T_{f})$ where $T_{f}$ is the natural isomorphism
\[
f^{\ast} \circ T_{-/Y} \xrightarrow{\cong} T_{((-) \times_Y X)/X}.
\]
\end{proposition}

We will now, as an application to Theorem \ref{Thm: Ind Tangent Category} above, construct the tangent category of formal schemes over a base scheme $S$.

\begin{example}\label{Example: Formal Tangent Schemes}
Let $S$ be an arbitrary base scheme and consider the Zariski tangent category $(\Sch_{S},T_{\Zar{S}})$ of $\Sch_{/S}$. Fix a Grothendieck universe $\Uscr$ and let $\Ind_{\Uscr}(\Cscr)$ be the ind-category of $\Cscr$ where each object $\ul{X}$ is indexed by a filtered $\Uscr$-small category (cf.\@ \cite[I.8.2.4.5]{SGA4}). Then by \cite[Definition 4.5]{Strickland}\footnote{It is worth remarking that there are \emph{various} different definitions of formal schemes in the literature. For instance, in \cite[Definition 10.4.2]{EGA1} Grothendieck and Dieudonn{\'e} define formal schemes in terms of topologically ringed spaces; in \cite{Yasuda} Yasuda defines formal schemes as certain proringed spaces, i.e., as a certain type of topological spaces equipped with a specific flavour of sheaf of prorings; and in \cite{hartshorne1977algebraic} are defined in terms of completions along only Noetherian schemes. What matters, however, is that each such situation gives a subcategory of $\Ind(\Sch_{/S})$ and so we take the more purely categorical perspective in this paper.} $\Ind_{\Uscr}(\Sch_{/S}) = \FSch_{/S}$, i.e., the $\Uscr$-small ind-category of $\Sch_{/S}$ is equivalent to the category of formal schemes over $S$. Applying  Theorem \ref{Thm: Ind Tangent Category} to $\Ind_{\Uscr}(\Sch_{/S})$ we find that $\FSch_{/S}$ is a tangent category with tangent formal scheme described as follows. If $\Xfrak = (X_i)_{i \in I} = \colim_{i \in I} X_i$ is a formal scheme over $S$ then the Zariski tangent formal scheme of $\Xfrak$ is the formal scheme
\[
T_{\Xfrak/S} := (T_{X_i/S})_{i\in I} = \colim_{i \in I} T_{X_i/S}.
\]
\end{example}

A straightforward formal consequence of our main theorem and Proposition \ref{Prop: Functoriality of Scheme Morphisms} is that when we have a morphism of schemes $f:S \to Y$, the pullback functor $\ul{f}^{\ast}:\FSch_{/Y} \to \FSch_{/S}$ is part of a strong tangent morphism. 
\begin{proposition}\label{Prop: Pullback of relative formal schmes over scheme base is strong tangent}
Let $f:S \to Y$ be a morphism of schemes. Then the pullback functor $f^{\ast}:\FSch_{/Y} \to \FSch_{/S}$ is part of a strong tangent morphism.
\end{proposition}
\begin{proof}
We first note that the functor $f^{\ast}:\FSch_{/Y} \to \FSch_{/S}$ is define by sending a formal scheme $\Xfrak \to Y$ to the formal scheme $\Xfrak \times_Y S$ (and similarly on morphisms). However, if $\Xfrak = (X_i)_{i \in I}$ then $\Xfrak \times_{Y} S$ is represented by the Ind-object $(X_i \times_Y S)_{i \in I}$. Consequently we define our natural isomorphism
\[
(T_{f}):f^{\ast} \circ T_{-/Y} \to T_{-/S} \circ f^{\ast}
\]
by defining the $\Xfrak$-component
\[
(T_f)_{\Xfrak}:(f^{\ast} \circ T_{-/Y})(\Xfrak) \to (T_{-/S} \circ f^{\ast})(\Xfrak)
\]
to be the map
\[
(T_f)_{\Xfrak} = \left((T_f)_{X_i}\right)_{i \in I}
\]
where each $(T_f)_{X_i}$ is the $X_i$-component of the natural isomorphism $T_f$ of Proposition \ref{Prop: Functoriality of Scheme Morphisms}. That this is a morphism of $\Ind$-schemes is trivial to check and that it is an isomorphism is immediate as well. Finally that the pair $(f^{\ast}, (T_f))$ is a strong tangent morphism follows from the fact that the tangent morphism identities and strengths are checked locally.
\end{proof}

The more general situation (determining whether or not whether or not the pullback functor $\ul{f}^{\ast}:\FSch_{/\Yfrak} \to \FSch_{/\Xfrak}$ is a strong tangent morphism) is significantly more complicated problem. While I anticipate that it is true, and that this can and should use Proposition \ref{Prop: A pseudolimit thing} to handle pullback along a morphism $\ul{f}:X \to \Yfrak$, establishing this in complete generality is left as future work.
We close this subsection with some explicit examples of the formal tangent scheme of some various formal schemes over various bases.
\begin{example}
Let $K$ be a field (of characteristic zero, for simplicitly) and let $S := \Spec K$ be our base scheme. Now consider the scheme $X = \Spec K[t]$ and note that there is a dense embedding of rings
\[
K[t] \hookrightarrow \lim_{n \in \N}\frac{K[t]}{(t^{n})}.
\]
Now consider the formal scheme $\Xfrak := \Spf K\llbracket t \rrbracket$, which we define to be the $\Ind$-object
\[
\Spf K\llbracket t \rrbracket := \left(\Spec\frac{K\llbracket t \rrbracket}{(t^n)}\right)_{n \in \N} \cong \left(\Spec\frac{K[t]}{(t^n)}\right)_{n \in \N}.
\]
Then $\Spf K\llbracket t \rrbracket$ gives us a nilpotent thickening of $\Spec K[t]$ at the origin which is \emph{not} equivalent to a scheme. We can describe the formal tangent scheme $T_{\Spf K\llbracket t \rrbracket/K}$ as follows. By construction we have that
\begin{align*}
T_{\Spf K\llbracket t \rrbracket/\Spec K} &= T_{-/\Spec K}\left(\Spec\frac{K[t]}{(t^n)}\right)_{n \in \N} = \left(T_{\Spec(K[t]/(t^n))/\Spec K}\right)_{n \in \N} \\
&\cong \left(\Spec\left(\Sym\left(\Omega^1_{(K[t]/(t^n))/K}\right)\right)\right)_{n \in \N} \cong \Spec\left(\Sym\left(\frac{K[t]\mathrm{d}t}{(t^n,\mathrm{d}(t^n)/\mathrm{d}t)}\right)\right)_{n \in \N} \\
&\cong \left(\Spec\left(\Sym\left(\frac{K[t]\mathrm{d}t}{(t^n, nt^{n-1}\mathrm{d}t)}\right)\right)\right)_{n \in \N}.
\end{align*}
\end{example}

We now close by making some short conjectures regarding the tangent category of formal schemes.

\begin{conjecture}
Let $\Xfrak$ and $\Yfrak$ be formal schemes over a base scheme $S$ and let $\ul{f}:\Xfrak \to \Yfrak$ be a morphism. Then the pullback functor $\ul{f}^{\ast}:\FSch_{/\Yfrak} \to \FSch_{/\Xfrak}$ is part of a strong tangent morphism.
\end{conjecture}
\begin{conjecture}
If every morphism appearing in a formal scheme $\Xfrak = (X_i)_{i \in I}$ is a closed immersion then the same is true of $T_{\Xfrak/S}$.
\end{conjecture}

\subsection{The Ind-Tangent Category of the CDC of Polynomials}
In this subsection we fix a crig $A$. An important and well-known $A$-linear CDC (cf.\@ \cite[Example 2.7.a]{JSproperties}) is the CDC $\APoly$ of polynomials  with coefficients in $A$. This is the category defined as follows:
\begin{itemize}
	\item Objects: $n \in \N$;
	\item Morphisms: A map $\varphi:n \to m$ is given by an $m$-tuple of polynomials with scalars in $A$ with $n$-variables. That is,
	\[
	\varphi = \big(p_{1}(x_1, \cdots, x_n), \cdots, p_{m}(x_1, \cdots, x_{n})\big), \qquad p_j(x_1, \cdots, x_n) \in A[x_1, \cdots, x_n].
	\]
	\item Composition: The composition of morphisms $\varphi:n \to m$ and $\psi:m \to \ell$, if $\varphi$ is the $m$-tuple $(f_1(x_1, \cdots, x_n), \cdots, f_m(x_1, \cdots, x_n))$ and $\psi$ is the $\ell$-tuple $(g_1(x_1, \cdots, x_m), \cdots, g_{\ell}(x_1, \cdots, x_m))$, then $\psi \circ \varphi$ is defined by
	\[
	\psi \circ \varphi := \bigg(g_1\big(f_1(x_1, \cdots, x_n),\cdots, f_m(x_1,\cdots, x_m)\big), \cdots, g_{\ell}\big(f_1(x_1, \cdots, x_n), \cdots, f_m(x_1, \cdots, x_n)\big)\bigg).
	\]
	\item Identities: The identity map $\id_n:n \to n$ is given by
	\[
	\id_n := (x_1, \cdots, x_n).
	\]
\end{itemize}
In this category the product of $n$ with $m$ is the natural number sum $n + m$ (analogously to how the dimension of affine space satisfies $\Abb_{K}^n \times_{\Spec K} \Abb_{K}^{m} \cong \Abb_{K}^{n+k}$ in the scheme-theoretic world; note here $K$ is a cring). The differential combinator on $\APoly$ takes an $m$-tuple $\varphi:n \to m$ with
\[
\varphi = \big(p_{1}(x_1, \cdots, x_n), \cdots, p_{m}(x_1, \cdots, x_{n})\big)
\]
and sends it to the tuple $D(\varphi):n \times n \to m$ given by the sum of its formal total derivatives:
\[
D(\varphi) := \left( \left(\sum_{i=1}^{n}\frac{\partial\,p_1(x_1, \cdots, x_n)}{\partial\,x_i}\right)y_i, \cdots, \left(\sum_{i=1}^{n}\frac{\partial\,p_m(x_1, \cdots, x_m)}{\partial\,x_i}\right)y_i\right).
\]
By \cite{JSproperties} we know that a morphism $\varphi:n \to m$ is $D$-linear if and only if $\varphi = (p_1, \cdots, p_m)$ is comproised of degree one monomials. In particular, such a map induces an $A$-linear transformation $A^n \to A^m$ and there is an equivalence of categories (cf.\@ \cite{JSproperties}) $\APoly_{\Dlin} \simeq \AModfd$ where $\AModfd$ denotes the category of finite dimensional $A$-modules.

Let us now give a concrete description of $\Ind(\APoly)$. The objects here are, of course, filtered diagrams $I \to \APoly$ and so every object can be represented as a filtered diagram of natural numbers and tuples of polynomials. By Proposition \ref{Prop: Ind of CDC is not a CDC} and the description of the linear maps above, we see that $\Ind(\APoly)$ is not a CDC. However, we will now characterize the differential objects in $\Ind(\APoly)$ and then use this to prove that the category of such objects is equivalent to the category of $A$-modules.

\begin{Theorem}\label{Thm: Diff Obs in APoly are Modules}
The category of differential objects in $\Ind(\APoly)$ together with linear bundle maps between them is equivalent to $\Ind(\AModfd)$. In particular, there is an equivalence of categories of $\mathbf{Diff}(\Ind(\APoly))$ and the category $\AMod$ of $A$-modules.
\end{Theorem}
\begin{proof}
By Corollary \ref{Cor: Dlin nes} we find that a differential object in $\Ind(\APoly)$ is necessarily a filtered diagram $F:I \to \APoly$ such that for every morphism $\varphi \in I_1$, $F(\varphi)$ is $D$-linear in $\APoly$. Consequently, $F$ factors as $F:I \to \APoly_{\Dlin} \to \APoly$ and so we induce an equivalence of categories
\[
\mathbf{Diff}(\Ind(\APoly)) \simeq \Ind(\APoly_{\Dlin}) \simeq \Ind(\AModfd).
\]
Now, because the category of $A$-modules is compact (in the sense that every $A$-module is a filtered colimit of its finite dimensional submodules) and cocomplete it follows from \cite[Corollary 6.3.5]{KashiwaraSchapira} that $\Ind(\AModfd) \simeq \AMod$. Chaining the equivalences together gives
\[
\mathbf{Diff}(\Ind(\APoly)) \simeq \Ind(\APoly_{\Dlin}) \simeq \Ind(\AModfd) \simeq \AMod
\]
as was desired.
\end{proof}
\begin{corollary}
The Ind-tangent category $\Ind(\APoly)$ contains a sub-CDC equivalent to the category $\AMod$ of $A$-modules.
\end{corollary}

As an aside, note that the category $\Ind(\APoly)$ also contains an object representing polynomials in any number of variables. If $I$ is a set, regard $I$ as a discrete category. Then the object $F:I \to \APoly$ given by $F(i) = 1$. Then $F$ represents the coproduct
\[
\bigotimes_{i \in I} A[x] \cong A[x_i:i \in I]
\]
and so we get that $\Ind(\APoly)$ contains polynomial objects with arbitrary numbers of variables.

\subsection{The Ind-Tangent Category of the CDC of Smooth Maps}\label{Subsection: Smooth}
We close this section of the paper by computing some facts about the CDC $\Smooth$ of smooth functions between Euclidean spaces (cf.\@ \cite{JSproperties}). This is a CDC which is important for modeling the theory of smooth functions between manifolds and other differential geometric contexts. This category is defined as follows:
\begin{enumerate}
	\item Objects: Euclidean spaces $\R^n$ for $n \in \N$;
	\item Morphisms: Smooth functions $f:\R^n \to \R^m$;
	\item Composition and Identities: As in $\Set$.
\end{enumerate}
Note that $\Smooth$ is an $\R$-linear category as each hom-set satisfies
\[
\Smooth(\R^n, \R^m) = \mathcal{C}^{\infty}(\R^n,\R^m)
\]
and so is canonically a generically infinite-dimensional $\R$-vector space\footnote{The exception lies with $\R^0$.} with bilinear composition. The differential combinator $D$ on $\Smooth$ is defined by sending a function $f:\R^n \to \R^m$ to its total derivative (viewed as a function $D(f):\R^n \times \R^n \to \R^m$). Explicitly, begin by noting that a smooth function $f:\R^n \to \R^m$ can be seen as an $m$-tuple of smooth functions $f_i:\R^n \to \R$, i.e.,
\[
f = (f_1, \cdots, f_m):\R^n \to \R^m.
\]
Writing $D_i(f)$ for the function $\partial f/\partial x_i$, we then define $D(f):\R^n \times \R^n \to \R^m$ by setting
\[
D(f)(\mathbf{v}, \mathbf{w}) := \left(\sum_{i=1}^{n}D_i(f_1)(\mathbf{v})\mathbf{w}_i, \cdots, \sum_{i=1}^{n}D_i(f_m)(\mathbf{v})\mathbf{w}_i\right).
\]
In this way we find, following \cite{JSproperties}, that a morphism $f:\R^n \to \R^m$ is $D$-linear if and only if $f$ is linear in the usual sense. Similarly to the case for $\APoly_{\Dlin}$, this gives rise to an equivalence of categories $\Smooth_{\Dlin} \simeq \RVectfd$ where $\RVectfd$ is the category of finite dimensional real vector spaces.

We will proceed as in the $\APoly$ case and determine the differential bundles over $\top_{\Ind(\Smooth)}$. Perhaps unsurprisingly, we will show that $\mathbf{Diff}(\Ind(\Smooth)) \simeq \RVect$ before making some closing remarks regarding some of the objects which we can find in $\Ind(\Smooth)$.
\begin{Theorem}\label{Thm: Diff objects in IndSmooth are Real Vecs}
The category of differential objects in $\Ind(\Smooth)$ together with linear bundle maps between them is equivalent to $\Ind(\RVectfd)$. In particular, there is an equivalence of categories between $\mathbf{Diff}(\Ind(\Smooth))$ and the category $\RVect$ of $\R$ vector spaces.
\end{Theorem}
\begin{proof}
The proof follows mutatis mutandis to that of Theorem \ref{Thm: Diff Obs in APoly are Modules}.
\end{proof}

We conclude by discussing some of the objects which we can find in $\Ind(\Smooth)$. Let $M$ be an arbitrary unbounded infinite-dimensional smooth manifold whose finite dimensional subspaces are Euclidean spaces. Then $M$ represents an object in $\Ind(\Smooth)$ by taking the filtered category $I$ to be the lattice of finite dimensional subspaces and the functor $F:I \to \Smooth$ sends each subspace to its Euclidean representative. Consequently we can find infinite-dimensional locally Euclidean manifolds in $\Ind(\Smooth)$ and so we can use $\Ind(\Smooth)$ as a tangent category for functional analytic settings.

\bibliography{ETCBib}
\bibliographystyle{amsplain}

\end{document}